\documentclass[12pt]{article}
\usepackage{}
\usepackage{amssymb}
\usepackage{mathrsfs}
\usepackage{amsfonts}
\usepackage{amsmath}
\usepackage{amsthm}
\usepackage{amssymb,url}
\usepackage{graphicx}
\usepackage{epstopdf}
\usepackage{cite}
\newtheorem{thm}{Theorem}[section]

\newtheorem{lem}[thm]{Lemma}
\newtheorem{rem}{Remark}[section]

\newtheorem{pro}{Proposition}[section]

\theoremstyle{definition}
\newtheorem{definition}{Definition}[section]

\makeatletter\@addtoreset{equation}{section}
\renewcommand\theequation
{\ifnum \c@section>\z@ \thesection.\fi\@arabic\c@equation}
\makeatother

\begin{document}

\title{An affine Orlicz P\'olya-Szeg\" o principle }
\author{Youjiang Lin
\footnote{Research of the author is supported by the funds of the Basic and Advanced Research Project of CQ CSTC cstc2015jcyjA00009 and Scientific and Technological Research Program of Chongqing Municipal Education Commission (KJ1500628),  Scientific research funds of Chongqing Technology and Business University 2015-56-02.}
}
\date{}
 \maketitle
{\bf Abstract}. The author \cite{LYJ17} established the affine Orlicz P\'olya-Szeg\" o principle for log-concave functions and conjectured that the principle can be extended to the general Orlicz Sobolev functions. In this paper, we confirm this conjecture completely.  An affine Orlicz P\'olya-Szeg\" o principle, which includes all the previous affine P\'olya-Szeg\" o principles as special cases, is formulated and proved.  As a consequence, an Orlicz-Petty projection inequality for star bodies is established.



\section{Introduction}

The classical P\'olya-Szeg\" o principle \cite{PS51} states that, given a non-negative function $f:\mathbb{R}^n\rightarrow \mathbb{R}$, the Dirichlet integral
$\int_{\mathbb{R}^n}|\nabla f|^p$ decreases under suitable rearrangements, the two most common of which are the symmetric decreasing rearrangement about a point and Steiner symmetrization about a hyperplane. Their corresponding P\'olya-Szeg\"o inequalities are a powerful tool to approach a wide number of variational problems of geometric and functional nature (see, e.g.,\cite{BCFP,Bu97,BF15,CPS15,CEFT08,CF06,ET04,Wang12,Wang13}).

Zhang \cite{Zhang99} and Lutwak et al. \cite{LYZ02} formulated
and proved a remarkable affine  $L_p$ P\'olya-Szeg\" o principle for $1\leq p<n$, which significantly strengthens the classical P\'olya-Szeg\" o principle. Later,  Cianchi et al. \cite{CLYZ09} perfectly completed  the picture of affine P\'olya-Szeg\" o principle (for general $p\geq 1$). In \cite{HSX12}, Haberl, Schuster and Xiao  obtained a beautiful asymmetric version of the affine P\'olya-Szeg\"o principle and proved that it is stronger and directly implies the symmetric form of Cianchi et al.

 The affine $L_p$ P\'olya-Szeg\" o-type principle is closely related to the $L_p$ Brunn-Minkowski theory of convex bodies (see e.g.,
 \cite{BLYZ12,Gardner02,HS09,HS0902,HSX12,LL11,Ludwig05,Ludwig14,LR10,LXZ11,Lu93,Lu96,LYZ00,LYZ06,Werner08} for additional references). The fact that $L_p$ spaces have a  natural generalization, known as Orlicz spaces, motivated Lutwak, Yang and Zhang \cite{LYZ10,LYZ1002} to initiate an extension of the $L_p$ Brunn-Minkowski theory to an Orlicz-Brunn-Minkowski theory. The definition of a corresponding addition came later, in the work of Gardner, Hug and Weil \cite{GHW14}. They developed a very general and comprehensive framework for Orlicz-Brunn-Minkowski theory. This theory has expanded rapidly (see e.g., \cite{Boroczky13,GHW15,GZ98,HP14,Lijin16,XJL14,Ye16,ZXZ14,Zhu12,ZX14}).

Let $\mathcal{N}$ be the class of convex functions $\phi:\mathbb{R}\rightarrow [0,\infty)$ such that $\phi(0)=0$ and such that $\phi$ is either strictly decreasing on $(-\infty,0]$ or $\phi$ is strictly increasing on $[0,\infty)$.
 Moreover, we will assume throughout that $\Phi(t)=\max\{\phi(t),\phi(-t)\}$, $t\in[0,\infty)$, and $\lim_{t\rightarrow+\infty}\Phi(t)/t=+\infty$. It is easily checked that $\Phi(t)$ is a convex function and  strictly increasing on $[0,\infty)$.

 Let $G$ denote a bounded open subset of $\mathbb{R}^n$ and $W^{1,\Phi}(G)$ the first order Orlicz-Sobolev space on $G$ (see Section \ref{s2}). Let $W_0^{1,\Phi}(G)$ denote the subspace of $W^{1,\Phi}(G)$ of those functions whose continuation by $0$ outside $G$ belongs to $W^{1,\Phi}(\mathbb{R}^n)$. For $v\in S^{n-1}$ and $f\in  W_0^{1,\Phi}(G)$,  we define
\begin{eqnarray}\label{52}
\|v\|_{f,\phi}=\|\nabla_vf\|_{\phi}=\inf\left\{\lambda>0:\;\frac{1}{|G|}\int_{G}\phi\left(\frac{\nabla_vf}{\lambda}\right)dx\leq 1\right\},
\end{eqnarray}
where $\nabla_vf$ is the directional derivative of $f$ in the direction $v$. The definition immediately provides the extension of $\|\cdot\|_{f,\phi}$ from $S^{n-1}$ to $\mathbb{R}^n$. Now $(\mathbb{R}^n,\|\cdot\|_{f,\phi})$ is the $n$-dimensional Banach space that we shall associate with $f$.
And its unit ball $B_{\phi}(f)=\{x\in\mathbb{R}^n:\|x\|_{f,\phi}\leq 1\}$ is a convex body in $\mathbb{R}^n$. An important fact is that its volume $|B_{\phi}(f)|$ is invariant under affine transformations of the form $x\mapsto Ax+x_0$, with $x_0\in\mathbb{R}^n$ and $A\in SL(n)$. We call the unit ball $B_{\phi}(f)$ the {\it Orlicz-Sobolev affine ball of $f$}. We call
\begin{eqnarray}\label{1a}
\mathcal {E}_{\phi}(f):=|B_{\phi}(f)|^{-\frac{1}{n}}=\left(\frac{1}{n}\int_{S^{n-1}}\|\nabla_vf\|_{\phi}^{-n}dv\right)^{-\frac{1}{n}}
\end{eqnarray}
the {\it Orlicz-Sobolev affine energy} of $f$.

Talenti \cite{Ta94} (also see \cite{BZ88}) proved a Euclidean Orlicz P\'olya-Szeg\" o principle,
\begin{eqnarray}\label{1d}
\|\nabla f^{\star}\|_{\phi}\leq \|\nabla f\|_{\phi},
\end{eqnarray}
 which extended the classical P\'olya-Szeg\" o principle to Orlicz-Sobolev spaces.

 In \cite{LYJ17}, the author established an affine Orlicz P\'olya-Szeg\"o principle for log-concave function: If $f\in W_0^{1,\Phi}(G)$ is a log-concave functions and $f^{\star}$ is its symmetric decreasing rearrangement, then
 \begin{eqnarray}\label{a2}
 \mathcal{E}_{\phi}(f^{\star})\leq \mathcal{E}_{\phi}(f),
 \end{eqnarray}
which includes all the affine $L_p$ P\'olya-Szeg\"o principles as special cases (when restricted to log-concave functions). The author \cite{LYJ17} conjectured that the  principle can be extended to the general Orlicz Sobolev functions. In this paper, we confirm this conjecture completely.
\begin{thm}\label{1}
 If $f\in W_0^{1,\Phi}(G)$  and $f^{\star}$ is its symmetric decreasing rearrangement, then
\begin{eqnarray}\label{a1}
\mathcal {E}_{\phi}(f^{\star})\leq \mathcal {E}_{\phi}(f).
\end{eqnarray}
\end{thm}

When $\phi(t)=(1-\lambda)(t)_+^p+\lambda(t)_-^p$, where $p>1$, $\lambda\in[0,1]$, $(t)_+:=\max\{t,0\}$ and $(t)_-:=\max\{-t,0\}$, the affine Orlicz P\'olya-Szeg\" o principle becomes the general affine P\'olya-Szeg\" o-type principle (established in \cite{Nguyen16}). The symmetric affine P\'olya-Szeg\" o principle \cite{CLYZ09} and the asymmetric affine P\'olya-Szeg\" o principle \cite{HSX12} correspond to the cases of $\lambda=1/2$ and $\lambda=0$, respectively.

\section{Outline of the proof}\label{s1}

The proof of the  affine  $L_p$ P\'olya-Szeg\" o principle in \cite{LYZ02} mainly  relies on the affine $L_p$ Petty projection inequality  from \cite{LYZ00} and  the solution of the normalized $L_p$ Minkowski problem \cite{Lu93}.  Haberl, Schuster and Xiao in \cite{HSX12} mainly relies on the Haberl-Schuster version of the affine isoperimetric inequality \cite{HS09} and the solution of the normalized $L_p$ Minkowski problem \cite{Lu93} to prove the asymmetric affine P\'olya-Szeg\" o principle. The discovery of the Haberl-Schuster version of the $L_p$ Petty projection inequality was made possible through the important advances in valuation theory by Ludwig \cite{Ludwig02,Ludwig05}. Inspired by the work of Haddad, Jim\'enez, and Montenegro \cite{HJM16},  Nguyen \cite{Nguyen16} made use of the general $L_p$ Busemann-Petty centroid inequality \cite{LYZ00} to prove the general affine P\'olya-Szeg\" o-type principle.

As mentioned in \cite{LYJ17}, since the function $\phi$ defining the Orlicz-Sobolev spaces is usually not multiplicative, i.e., $\phi(xy)\neq\phi(x)\phi(y)$ for $x,y\in\mathbb{R}$, and the Orlicz Minkowski problem has not been completely solved, we can not use the Orlicz-Petty projection inequality \cite{LYZ10} and the solution of the even Orlicz Minkowski problem \cite{HLYZ10} to prove the affine Orlicz P\'olya-Szeg\" o principle.  In \cite{LYJ17}, making use of the  functional Steiner symmetrization, the author gave a direct proof of the Orlicz P\'olya-Szeg\" o principle for log-concave functions.  In this paper, we still make use of the functional Steiner symmetrization to the affine Orlicz P\'olya-Szeg\" o principle. Comparing with the paper \cite{LYJ17}, in this paper we mainly overcome the following several difficulties.

Firstly, in \cite{LYJ17}, the convexity property of the level sets of a log-concave function plays a role in proving the Orlicz P\'olya-Szeg\" o principle. But for the general $W_0^{1,\Phi}(G)$  functions, the level sets of functions are sets more general than convex bodies. Thus, it is crucial to figure out the structure of the level sets of $W_0^{1,\Phi}(G)$. In this paper, we show that the level sets of $W_0^{1,\Phi}(G)$  functions are sets of finite perimeter and define the Steiner symmetrization of the level sets in a different way (see Section \ref{ss1}), which makes it possible to make use of the functional Steiner symmetrization to prove the affine Orlicz P\'olya-Szeg\" o principle.

Secondly, since log-concave functions are co-area regular, by \cite[Theorem 2]{Bu97}, the symmetric decreasing rearrangements of log-concave functions can be approximated in the strong $W^{1,p}$-topology by a sequence of Steiner symmetrizations. Therefore, the Orlicz-Sobolev affine ball operator $B_{\phi}:W_0^{1,\Phi}(G)\rightarrow \mathcal{K}_o^n$ is continuous for log-concave functions.  But for the general Sobolev functions, the Orlicz-Sobolev affine ball operator may be discontinuous for the co-area irregular functions. In this paper, making use of  the weak convergence of Steiner symmetrizations of $f\in W^{1,\phi}(\mathbb{R}^n)$, we prove that the Orlicz-Sobolev affine ball operator is weakly continuous in such a way that
 there exist a convex body $K_0$ and a sequence of successive Steiner symmetrizations $\{f_k\}_{k\geq 0}$ of $f$ such that $B_{\phi}(f_i)$ converges to $K_0$ and $K_0\subset B_{\phi}(f^{\star})$ (see Lemma \ref{72}).

Thirdly, in \cite[Lemma 5.2]{LYJ17}, for log-concave functions, the author proved that $\nabla f$ and $\nabla Sf$ are equal to the same constant in the line parallelling to $e_n$ on the boundaries of subgraphs of $f$ and $Sf$. For general Sobolev functions, without the convexity, we can not make use of Lemma 1.5.14 and Theorem 1.5.15 in \cite{Schneider13} with respect to convex functions to prove the similar conclusion. In this paper, using the more subtle method, we prove that Sobolev functions have the similar properties as log-concave functions (see Lemma \ref{1c}).

Functional Steiner symmetrization has turned out to be very fruitful in proving isoperimetric theorems in analysis and function theory (see. e.g., \cite{Bianchi17,Bu97,Bu04,Capriani14,Ci00,CCF05,CF06,CF0602,FMP08,Lin17,LYJ17,Trudinger97}, and the references therein). In the beautiful paper \cite{CF06}, Cianchi and Fucso analyzed the case of equality  in Steiner symmetrization inequalities for Dirichlet-type integrals. Specially, Cianchi and Fucso \cite{CF06} proved the following P\'olya-Szeg\" o principle for $f\in W_0^{1,1}(G)$:
$$\int_{SG}\psi(\nabla Sf)dx\leq \int_{G}\psi(\nabla f)dx,$$
where $Sf$ is the Steiner symmetrization of $f$ and $\psi$ is a convex function from $\mathbb{R}^n$ into $[0,+\infty)$, vanishing at $0$. Comparing with Cianchi-Fucso's method---using the Steiner symmetrization of one-dimensional restrictions of Sobolev functions and Fubini's theorem, we mainly make use of the Steiner symmetrization of the level sets and the co-area formula.

\section{Notation and preliminary results}\label{s2}

 The setting will be Euclidean $n$-space $\mathbb{R}^{n}$ with origin $o$.  We write $e_1,\dots, e_{n-1},e_n$ for the standard orthonormal basis of $\mathbb{R}^n$ and when we write $\mathbb{R}^n=\mathbb{R}^{n-1}\times\mathbb{R}$ we always assume that $e_n$ is associated with the last factor.

 We will attempt to use $x, y$ for vectors in $\mathbb{R}^n$ and $x^{\prime}$, $y^{\prime}$ for vectors in $\mathbb{R}^{n-1}$, and $u,v\in S^{n-1}$ for unit vectors in $\mathbb{R}^n$.  Let $u^{\perp}$ denote the $n$-dimensional linear subspace orthogonal to $u$ in $\mathbb{R}^n$.  Let $l_{u}$ denote the  line that passes through the origin and is parallel to $u$.  We will use $a,b,s,t,\alpha$ for numbers in $\mathbb{R}$ and $c,\lambda$ for strictly positive reals. If $Q$ is a Borel subset of $\mathbb{R}^n$ and $Q$ is contained in an $i$-dimensional affine subspace of $\mathbb{R}^n$ but in no affine subspace of lower dimension, then
$|Q|$ will denote the $i$-dimensional Lebesgue measure of $Q$.  Let $\mathcal{H}^n$ denote the $n$-dimensional Hausdorff measures. Let $B_n$ denote the Euclidean unit ball in $\mathbb{R}^n$. Let $\omega_n$ denote the $n$-dimensional volume of the unit ball in $\mathbb{R}^n$.  If $x\in\mathbb{R}^n$ then by abuse of notation we will write $|x|=\sqrt{x\cdot x}$.

For $A\in SL(n)$ write $A^t$ for the transpose of $A$ and $A^{-t}$ for the inverse of the transpose of $A$. We write $|A|$ for the absolute value of the determinant of $A$.

\subsection{Convex bodies and star bodies}

In this section we fix our notation and collect basic facts from convex geometry.
General references for the theory of convex bodies are the books by  Gardner \cite{Ga06}, Gruber \cite{Gruber07}, Schneider \cite{Schneider13}. We write $\mathcal{K}^n$ for the set of convex bodies (compact convex subsets) of $\mathbb{R}^n$. We write $\mathcal{K}^n_o$ for the set of convex bodies that contain the origin in their interiors. A compact set $K\subset\mathbb{R}^n$ is a star-shaped set (with respect to the origin) if the intersection of every straight line through the origin with $K$ is a line segment. Let $K\subset\mathbb{R}^n$ be a compact star shaped set (with respect to the origin), the {\it radial function} $\rho(K,\cdot):\mathbb{R}^n\backslash\{0\}\rightarrow\mathbb{R}$ is defined by
\begin{eqnarray}\label{2b}
\rho(K,x)=\rho_K(x)=\max\{\lambda\geq 0:\; \lambda x\in K\}.
 \end{eqnarray} If $\rho_K$ is strictly positive and continuous, then we call $K$ a star body (with respect to the origin). In what follows we will denote the class of star bodies (with respect to the origin) in $\mathbb{R}^n$ by $\mathcal{S}_o^n$.

For $K\in \mathcal{K}^n$, let $h(K;\cdot)=h_K:\mathbb{R}^n\rightarrow\mathbb{R}$ denote the {\it support function} of $K$; i.e.,
$$h(K;x):=\max\{x\cdot y:\;y\in K\}.$$
Thus, if $y\in\partial K$, then
\begin{eqnarray}\label{4a}
h_K(v_K(y))=v_K(y)\cdot y,
\end{eqnarray}
where $v_K(y)$ denotes an outer unit normal to $K$ at $y\in\partial K$.

For $K,L\in\mathcal{K}^n$, the {\it Hausdorff distance} of $K$ and $L$ is defined by
\begin{eqnarray}
\delta(K,L):=\sup_{u\in S^{n-1}}|h_{K}(u)-h_{L}(u)|.
\end{eqnarray}

If $K\in\mathcal {K}_o^n$, then the {\it polar body} $K^{\ast}$ is defined by
$$K^{\ast}:=\{x\in\mathbb{R}^n:x\cdot y\leq 1\;{\rm for\;all}\;y\in K\}.$$

For $K\in\mathcal{K}_o^n$, it is easily verified that
\begin{eqnarray}\label{41}
h_{K^{\ast}}=1/\rho_K\;\;\;{\rm and}\;\;\;\rho_{K^{\ast}}=1/h_K.
\end{eqnarray}

For $K\in \mathcal {S}_o^n$, the function $g_K(\cdot)=\|\cdot\|_K:\mathbb{R}^n\rightarrow[0,\infty)$ defined by
\begin{eqnarray}\label{41b}
g_K(x)=\|x\|_K=\inf\{\lambda>0:\;x\in\lambda K\}
\end{eqnarray}
is called the {\it gauge function} of $K$.

By (\ref{2b}) and (\ref{41b}), it is clear that for $K\in\mathcal{S}_o^n$,
\begin{eqnarray}\label{41a}
g_K=1/\rho_{K}.
\end{eqnarray}
By (\ref{41b}), for $x\in\mathbb{R}^n$ and $K\in\mathcal{S}_o^n$,  it follows immediately that
\begin{eqnarray}\label{41c}
g_K(x)=1\;\;{\rm if\;and\;only\;if}\;\;x\in\partial K.
\end{eqnarray}

Let $K\in\mathcal{S}_o^n$ be a star body whose radial function is positive and locally Lipschitz continuous.  Since $g_K=\frac{1}{\rho_K}$ and  there exist $0<c_1<c_2$ such that $c_1<\rho_K(u)<c_2$ for any $u\in S^{n-1}$, $g_K$ is positive and  locally Lipschitz continuous on $\mathbb{R}^n$. Therefore, for almost all $y\in \partial K$, the gradient $\nabla g_K(y)$ exists. For $y\in\partial K$, $g_K(y)=1$. Differentiating this equation gives $\nabla g_K(y)\cdot dy=0$. Thus, $\nabla g_K(y)$ is orthogonal to
the tangent hyperplane of $\partial K$ at $y$. For $r>0$, differentiating the equation $g_K(ry)=rg_K(y)$ with respect to $r$ gives
\begin{eqnarray}\label{2e}
y\cdot\nabla g_K(y)=g_K(y).
\end{eqnarray}
Therefore, $|\nabla g_K(y)|\neq 0$. The unit outer normal vector $v_K(y)$ to $K$ at $y\in\partial K$ is given by
\begin{eqnarray}\label{2d}
v_K(y)=\frac{\nabla g_K(y)}{|\nabla g_K(y)|}.
\end{eqnarray}
Moreover, for $y\in\partial K$,  by (\ref{2d}), (\ref{2e}) and $g_K(y)=1$, we have
\begin{eqnarray}\label{2f}
y\cdot v_K(y)=y\cdot \frac{\nabla g_K(y)}{|\nabla g_K(y)|}=\frac{1}{|\nabla g_K(y)|}.
\end{eqnarray}
By (\ref{2d}) and (\ref{2f}), for $y\in \partial K$, we have
\begin{eqnarray}\label{2h}
\nabla g_K(y)=\frac{v_K(y)}{y\cdot v_K(y)}.
\end{eqnarray}
Moreover, since that $g_K(ry)=rg_K(y)$ for $r>0$, setting $x=ry$, we obtain
\begin{eqnarray}\label{2i}
\nabla g_K(x)=\nabla g_K(y).
\end{eqnarray}

\subsection{$BV(\mathbb{R}^n)$ functions and sets of finite perimeter}
In this section, we review some basic definitions and facts about functions of bounded variation on $\mathbb{R}^n$.  Good general references for this are Ambrosio, Fusco and Pallara \cite{AFP00}, Evans and Gariepy\cite{EG92}, Ziemer\cite{Ziemer89}.

Let $C_c^1(\mathbb{R}^n)$ stand for the compactly supported continuously differentiable functions on $\mathbb{R}^n$. Let $f\in L^1(\mathbb{R}^n)$; we say that $f$ is a function of bounded variation of $\mathbb{R}^n$ if the weak derivative of $f$ is representable by a finite Radon measure on $\mathbb{R}^n$, i.e. if
$$\int_{\mathbb{R}^n}f\frac{\partial \phi}{\partial x_i}dx=-\int_{\mathbb{R}^n}\phi d D_i f\;\;{\rm for\;all\;}\phi\in C_{c}^{1}(\mathbb{R}^n),\;i=1,\dots,n$$
for some $\mathbb{R}^n$-valued measure $Df=(D_1f,\dots,D_nf)$ on $\mathbb{R}^n$. The vector space of all functions of bounded variation in $\mathbb{R}^n$ is denoted by $BV(\mathbb{R}^n)$.

A Borel set $E\subset\mathbb{R}^n$ is said to have {\it finite perimeter in an open set} $\Omega$ provided that the characteristic function of $E$, $\mathcal{X}_E$, is a function of bounded variation in $\Omega$. Thus, the partial derivatives of $\mathcal{X}_E$ are Radon measures in $\Omega$ and the perimeter of $E$ in $\Omega$, $P(E,\Omega)$, is defined as
$$P(E,\Omega)=\|D\mathcal{X}_E\|(\Omega).$$
A set $E$ is said to be of {\it locally finite perimeter} if $P(E,\Omega)<\infty$ for every bounded open set $\Omega$. If $E$ is of finite perimeter in $\mathbb{R}^n$, it is simply called a set of {\it finite perimeter}. Usually, we write $P(E)$ instead of $P(E,\mathbb{R}^n)$.

Let  $\mathcal{C}^n$ denote the set of  sets  of locally finite perimeter. By ${\rm cl}E$, ${\rm int}E$ and $\partial E$ we denote, respectively, the closure, interior and boundary of  $E\in \mathcal{C}^n$. Let $E^{\prime}_u$ denote the image of the orthogonal projection of $E$ onto $u^{\perp}$. The {\it essential projection} of a set $E\subset\mathcal{C}^n$ onto $u^{\perp}$ is defined as
$$E_u^{\prime,+}=\{x^{\prime}\in u^{\perp}:\;|(l_u+x^{\prime})\cap E|>0\}.$$

For $E,F\in \mathcal{C}^n$, the {\it symmetric difference} of $E$ and $F$ is
$$E\triangle F=(E\backslash F)\cup (F\backslash E).$$

Let $B(x,r)$ denote the ball, centered at $x$, having radius $r$. Let $E$ be of locally finite perimeter. The {\it reduced boundary} of $E$, $\partial^{\ast}E$, consists of all points $x\in\mathbb{R}^n$ for which the following hold:

(i) $\|D\mathcal{X}_E\|[B(x,r)]>0$ for all $r>0$,

(ii) If $v_{E,r}(x):=-D\mathcal{X}_E[B(x,r)]/\|D\mathcal{X}_E\|[B(x,r)]$, then the limit $v_E(x)=\lim_{r\rightarrow 0}v_{E,r}(x)$ exists with $|v_E(x)|=1$.

 $v_E(x)$ is called the {\it generalized exterior normal} to $E$ at $x\in\partial^{\ast}E$. By \cite[Theroem 3.59]{AFP00}, if $E\in \mathcal{C}^n$, $\partial^{\ast} E$ is $n-1$-rectifiable, i.e., it can be covered, except for an $\mathcal{H}^{n-1}$-negligible subset, by countably many $(n-1)$-dimensional surfaces of class $C^1$.

If $E\subset\mathbb{R}^n$ is a Lebesgue measurable set, the {\it measure-theoretic boundary of} $E$ is defined by
$$\partial^ME=\left\{x:\limsup_{r\rightarrow 0}\frac{|B(x,r)\cap E|}{r^n}>0\;{\rm and}\;\limsup_{r\rightarrow 0}\frac{|B(x,r)- E|}{r^n}>0\right\}.$$
By Federer's structure theorem (see, e.g.,  \cite[Lemma 1 on Page 208]{EG92} and \cite[Theroem 3.61]{AFP00}), if $E\in\mathcal{C}^n$, then $\partial^{\ast}E\subset\partial^M E$ and $\mathcal{H}^{n-1}(\partial^ME-\partial^{\ast}E)=0$.

For  $E\in\mathcal{C}^n$, let
$$D_E:=\sup\{|x-y|:\;x,y\in E\}$$
denote the {\it diameter} of $E$. Since $E$ is obviously contained in the right cylinder whose base is $E^{\prime}_u$ and whose height is $D_E$, then we have the crude estimate
\begin{eqnarray}\label{3e}
\int_{\partial^{\ast} E}(u\cdot v_E(y))_+d\mathcal{H}^{n-1}(y)\geq\frac{|E|}{D_E}.
\end{eqnarray}

\subsection{Orlicz-Sobolev spaces}

In this section we summarize the necessary definitions and results about Orlicz-Sobolev spaces.
 For a detailed account of these facts, the reader could consult the books of Maz'ya \cite{Ma85,Ma11} and the paper of Cianchi \cite{Ci10}.

Let $\phi\in \mathcal{N}$ and $\Phi(t)=\max\{\phi(t),\phi(-t)\}$, $t\in[0,\infty)$. Let $G$ be an open subset of $\mathbb{R}^n$. The Orlicz space $L^{\Phi}(G)$ is defined as
\begin{eqnarray}\label{33a}
L^{\Phi}(G)&=&\left\{f:\;f \;{\rm is\;a\;Lebesgue\;measurable\;real\;valued\;function\;on}\;G\right.\nonumber\\
&~&\left.{\rm such\;that}\;\int_{G}\Phi\left(\frac{|f(x)|}{\lambda}\right)dx<\infty\;{\rm for\;some\;}\lambda>0\right\}.
\end{eqnarray}

The Luxemburg norm $\|f\|_{L^{\Phi}(G)}$ is defined as
\begin{eqnarray}\label{33}
\|f\|_{L^{\Phi}(G)}=\inf\left\{\lambda>0:\int_{G}\Phi\left(\frac{|f(x)|}{\lambda}\right)dx\leq 1\right\}.
\end{eqnarray}
The space $L^{\Phi}(G)$, equipped with the norm $\|\cdot\|_{L^{\Phi}}$, is a Banach space. Note that, if $\Phi(s)=s^p$ and $p> 1$, then $L^{\Phi}(G)=L^p(G)$, the usual $L_p$ space, and $\|\cdot\|_{L^{\Phi}(G)}=\|\cdot\|_{L^{p}(G)}$. Usually, we write $\|\cdot\|_{\Phi}$ instead of $\|\cdot\|_{L^{\Phi}}$.

The first order Orlicz-Sobolev space $W^{1,\Phi}(G)$ is defined as
\begin{eqnarray}\label{6b}
W^{1,\Phi}(G)=\{f\in L^{\Phi}(G): f\;{\rm is}\;{\rm weakly}\;{\rm differentiable}\;{\rm and\;}|\nabla f|\in L^{\Phi}(G)\}.
\end{eqnarray}
Here, $\nabla$ denotes the weak gradient. The space $W^{1,\Phi}(G)$, equipped with the norm
\begin{eqnarray}
\|f\|_{W^{1,\Phi}(G)}=\|f\|_{\Phi}+\||\nabla f|\|_{\Phi},
\end{eqnarray}
is a Banach space. Clearly, $W^{1,\Phi}(G)=W^{1,p}(G)$, the standard Sobolev space, if $\Phi(s)=s^p$ with $p>1$.

$W_0^{1,\Phi}(G)$ will denote the subspace of $W^{1,\Phi}(G)$ of those functions whose extension by $0$ outside $G$ belongs to $W^{1,\Phi}(\mathbb{R}^n)$. If $f\in W_0^{1,\Phi}(G)$, then $f$ vanishes on the boundary of $G$ and we let
\begin{equation}\label{3a}
\bar{f}(x)=\left \{ \begin{aligned}
&f(x),&x\in G &\\
& 0,& x\in\mathbb{R}^n/G,&
 \end{aligned} \right.
 \end{equation}
 then $\bar{f}\in W^{1,\Phi}(\mathbb{R}^n)$.
Throughout this paper,  $\bar{f}$ will denote the extension of $f$ by $0$ outside $G$.

 For the nonnegative function $f\in W_0^{1,\Phi}(G)$, we define the {\it subgraph} of $f$ by
\begin{eqnarray}
{\rm sub}f:=\{(x,x_{n+1})\in\mathbb{R}^{n+1}:x\in G,\;0<x_{n+1}<f(x)\}.
\end{eqnarray}
And we define its {\it superlevel sets} by
\begin{eqnarray}
[f]_h:=\{x\in G:\;f(x)\geq h\},\;h>0.
\end{eqnarray}

\begin{lem}
If $f\in W_0^{1,\Phi}(G)$, then for $\mathcal{L}^1$-a.e. $h>0$, $[f]_h$ is a set of finite perimeter.
\end{lem}
\begin{proof}
By \cite [Lemma 3.3] {LYJ17}, if $f\in W_0^{1,\Phi}(G)$, then $\bar{f}\in W^{1,1}(\mathbb{R}^n)$. Since Sobolev space $W^{1,1}(\mathbb{R}^n)$ is the subspace of $BV(\mathbb{R}^n)$ (see e.g. \cite[Page 118]{AFP00}), $\bar{f}\in BV(\mathbb{R}^n)$. By \cite [Theorem 3.40]{AFP00}, if $f\in BV(\mathbb{R}^n)$, then the superlevel set $[f]_h$ is of finite perimeter for $\mathcal{L}^1$-a.e. $h>0$.
\end{proof}

We shall make use of the fact that for $\phi \in \mathcal{N}$, $a_i\in \mathbb{R}$ and $b_i>0$, $i=1,2,\cdots,m$,
\begin{eqnarray}\label{a7}
\sum_{i=1}^{m}b_i\cdot\phi\left(\frac{\sum_{i=1}^{m}a_i}{\sum_{i=1}^{m}b_i}\right)\leq \sum_{i=1}^{m}\left(b_i\phi\left(\frac{a_i}{b_i}\right)\right).
\end{eqnarray}
This is a trivial consequence of the convexity of $\phi$.

We shall make use of the following clear fact.
\begin{lem}\label{88}
If $\phi\in\mathcal{N}$, for $a,b\in\mathbb{R}$ and $a\neq 0$, then the function
\begin{eqnarray}\label{6}
\Psi(t):=\phi(at-b)+\phi(-at-b),\;t>0
\end{eqnarray}
is  increasing.
\end{lem}

\section{Steiner symmetrization}\label{s3}
\subsection{Steiner symmetrization of compact sets}

Without loss of generality, we only need to consider the  Steiner symmetrization in the direction $e_n$.

Given any compact subset $E$ of $\mathbb{R}^n$, define, for $x^{\prime}\in\mathbb{R}^{n-1}$,
\begin{eqnarray}
E_{x^{\prime}}:=\{x_n\in\mathbb{R}:(x^{\prime},x_n)\in E\}.
\end{eqnarray}
Hereafter, $\mathcal{L}^m$ denote the outer Lebesgue measure in $\mathbb{R}^m$ and $E^{\prime}$ denote the image of the orthogonal projection of $E$ onto $e_n^{\perp}$. Then, we define the {\it Steiner symmetral} $S E$ of $E$ about the hyperplane $e_n^{\perp}$ as
$$S E:=\{(x^{\prime},x_n)\in\mathbb{R}^n:x^{\prime}\in E^{\prime},\;|x_n|\leq \mathcal{L}^1(E_{x^{\prime}})/2\}.$$
Moreover, Steiner symmetrization preserves volume, i.e., $|SE|=|E|$.

When considering the convex body $K\in\mathcal {K}_o^n$ as $K\subset \mathbb{R}^{n-1}\times\mathbb{R}$,  the {\it Steiner symmetral}, $S K$, of $K$ in the direction $e_n$ is given by
\begin{eqnarray}\label{44}
SK=\left\{\left(x^{\prime},\frac{1}{2}t+\frac{1}{2}s\right)\in\mathbb{R}^{n-1}\times\mathbb{R}:(x^{\prime},t),(x^{\prime},-s)\in K\right\},
\end{eqnarray}
and its boundary is given by
\begin{eqnarray}\label{45}
\partial SK=\left\{\left(x^{\prime},\frac{1}{2}t+\frac{1}{2}s\right)\in\mathbb{R}^{n-1}\times\mathbb{R}:(x^{\prime},t),(x^{\prime},-s)\in\partial K\;{\rm with}\;t\neq-s\right\}.
\end{eqnarray}

In this paper, we shall make critical use of the following fact that follows directly from (\ref{44}), (\ref{45}), and (\ref{41c}).
\begin{lem}\label{46}
Suppose $K,L\in\mathcal {K}_o^n$ and consider $K,L\subset\mathbb{R}^{n-1}\times\mathbb{R}$. Then
$$SK\subset L,$$
if and only if
$$\|(x^{\prime},t)\|_K=1=\|(x^{\prime},-s)\|_{K},\;\; with\;t\neq-s\;\;\Longrightarrow \;\;\|(x^{\prime},t/2+s/2)\|_{L}\leq 1.$$
\end{lem}

\subsection{Steiner symmetrization of functions}

Let $\mathcal{X}_A(x)$ denote the characteristic function of $A\subset \mathbb{R}^n$, i.e.,
\begin{equation}
\mathcal{X}_A(x):=\left \{ \begin{aligned}
&1&&{\rm if}\;x\in A, \\
& 0&& {\rm if}\;x\in\mathbb{R}^n\backslash A.
 \end{aligned} \right.
 \end{equation}

\begin{definition}\label{36}
For nonnegative function $f\in W_0^{1,\Phi}(G)$,  the {\it Steiner symmetrization}, $S f$,  of $f$ with respect to $e_n$ is defined as
\begin{eqnarray}\label{2}
Sf(x)=\int_{0}^{\infty}\mathcal {X}_{S[f]_h}(x)dh, \;\;x\in SG.
\end{eqnarray}
\end{definition}

By Definition \ref{36}, it is easily checked that
\begin{eqnarray}\label{4c}
[Sf]_h\;{\rm is \; equivalent\;to}\;S[f]_h,\;\;{\rm for}\;h>0.
\end{eqnarray}
Thus, we have for the subgraphs
\begin{eqnarray}\label{3f}
{\rm sub}(Sf)\;{\rm is \; equivalent\;to}\;S({\rm sub}f).
\end{eqnarray}
The following lemma was given in \cite[Theorem 2.1]{CF06}.
\begin{lem}\label{3i}
If $f\in W_0^{1,\Phi}(G)$, then $Sf\in W_0^{1,\Phi}(SG)$.
\end{lem}

\subsection{The steiner symmetrization of the level sets of $W_0^{1,\Phi}(G)$ functions}\label{ss1}

For $f\in W_0^{1,\Phi}(G)$ and any $h>0$, $[f]_h$ is a set of finite perimeter. Let $[f]^{\prime}_h$ denote the image of orthogonal projection of $[f]_h$ onto $e^{\perp}_n$. For a fixed $x^{\prime}\in [f]_h^{\prime}$, let $[f]_h(x^{\prime})$ denote the set $\{x_n\in\mathbb{R}:(x^{\prime},x_n)\in [f]_h\}$. Let
\begin{eqnarray}\label{5a}
N([f]_h,x^{\prime}):=\frac{1}{2}\mathcal{H}^0\left(\partial^{\ast}([f]_h(x^{\prime}))\right),
\end{eqnarray}
where $\mathcal{H}^0$ is the counting measure. Since $[f]_h$ is a set of finite perimeter, by \cite[Theorem D]{CF06}, for $\mathcal{L}^{n-1}$-a.e. $x^{\prime}\in [f]^{\prime}_h$, $\mathcal{H}^0\left(\partial^{\ast}([f]_h(x^{\prime}))\right)$ is a finite even number.

For fixed positive integer $i$,  let
\begin{eqnarray}\label{5b}
[f]^{\prime}_{h,i}:=\{x^{\prime}\in [f]_h^{\prime}:\;N([f]_h,x^{\prime})=i\}.
\end{eqnarray}

Let
\begin{eqnarray}\label{5e}
[f]_{h,i}:=\{(x^{\prime},x_n)\in [f]_h:\;x^{\prime}\in [f]_{h,i}^{\prime}\}.
\end{eqnarray}

For the fixed $i$ and $k=1,2,\dots,i$,  let $\underline{\ell}^k_{h,i}:[f]_{h,i}^{\prime}\rightarrow \mathbb{R}$ and $\overline{\ell}^k_{h,i}:[f]_{h,i}^{\prime}\rightarrow \mathbb{R}$ denote the {\it $k$-th undergraph function and  $k$-th overgraph function} of $[f]_{h,i}$ with respect to $e_n$; i.e.,
\begin{eqnarray}
&&[f]_{h,i}=\bigcup_{k=1}^{i}\left\{(x^{\prime},x_n)\in\mathbb{R}^n:\;x^{\prime}\in [f]_{h,i}^{\prime},\;-\underline{\ell}^k_{h,i}(x^{\prime})\leq x_n\leq \overline{\ell}^k_{h,i}(x^{\prime})\right\}
\end{eqnarray}
up to an $\mathcal{L}^n$-negligible set.

 For the Steiner symmetral, $S[f]_{h,i}$, of $[f]_{h,i}$ in direction $e_n$, we see that the image of the orthogonal projections onto $e_n^{\perp}$ of both $[f]_{h,i}$ and $S[f]_{h,i}$ are identical.   Let $\overline{\varrho}_{h,i}$ and $\underline{\varrho}_{h,i}$ denote the overgraph and undergraph functions of $S[f]_{h,i}$ with respect to $e_n$, i.e.,
 \begin{eqnarray}
 S[f]_{h,i}=\left\{(x^{\prime},x_n)\in\mathbb{R}^n:\;x^{\prime}\in [f]_{h,i}^{\prime},\;-\underline{\varrho}_{h,i}(x^{\prime})\leq x_n\leq \overline{\varrho}_{h,i}(x^{\prime})\right\}.
 \end{eqnarray}

Then, by the definition of the Steiner symmetrization of compact sets, we have
\begin{eqnarray}\label{a6}
\underline{\varrho}_{h,i}(x^{\prime})=\overline{\varrho}_{h,i}(x^{\prime})=\frac{1}{2}\sum_{k=1}^{i}\left(\underline{\ell}^k_{h,i}(x^{\prime})+\overline{\ell}^k_{h,i}(x^{\prime})\right),\;\;\;x^{\prime}\in [f]^{\prime}_{h,i}.
\end{eqnarray}

For $f\in W_0^{1,\Phi}(G)$ and a Borel function $g:G\rightarrow \mathbb{R}$, the following  {\it co-area formula} (see, e.g., \cite[Theorem 1.2.4]{Ma85} or \cite[(3.7)]{CF06}) is an important tool in proving our main theorem.
\begin{eqnarray}\label{1e}
\int_{\{x\in G:\nabla f(x)\neq 0\}}g(x)dx=\int_{0}^{\infty}\int_{\partial^{\ast}[f]_h}g(x)|\nabla f(x)|^{-1}d\mathcal{H}^{n-1}(x)dh.
\end{eqnarray}
 Since $[f]_h$ is a set of finite perimeter, $\partial^{\ast}[f]_h$ consists of  an $\mathcal{H}^{n-1}$-negligible subset and countably many $(n-1)$-dimensional surfaces of class $C^1$.

  Let $$\bar{\partial}^{\ast} [f]_h=\left\{x\in\partial^{\ast} [f]_h:v_{[f]_h}(x)\cdot e_n=0\right\}$$
and
$$\breve{\partial}^{\ast}[f]_h=\left\{x\in\partial^{\ast} [f]_h: v_{[f]_h}(x)\cdot e_n\neq 0\right\}.$$
We will make use of the easily established fact that for a continuous function $g:\partial^{\ast} [f]_h\rightarrow\mathbb{R}$,
\begin{eqnarray}\label{53}
\int_{\breve{\partial}^{\ast} [f]_h}g(x)d\mathcal{H}^{n-1}(x)
&=&\sum_{i=1}^{\infty}\sum_{k=1}^{i}\int_{{\rm int}[f]_{h,i}^{\prime}}g(x^{\prime},\overline{\ell}^k_{h,i}(x^{\prime}))\sqrt{1+|\nabla\overline{\ell}^k_{h,i}(x^{\prime})|^2}dx^{\prime}\nonumber\\
&\;\;&+\sum_{i=1}^{\infty}\sum_{k=1}^{i}
\int_{{\rm int}[f]_{h,i}^{\prime}}g(x^{\prime},-\underline{\ell}^k_{h,i}(x^{\prime}))\sqrt{1+|\nabla\underline{\ell}^k_{h,i}(x^{\prime})|^2}dx^{\prime}\nonumber.\\
\end{eqnarray}

\subsection{Approximation of the symmetric decreasing rearrangement by Steiner symmetrizations}
Let $G$ be a bounded open set in $\mathbb{R}^n$. Its {\it symmetric rearrangement} $G^{\star}$ is the open centered ball whose volume agrees with that of $G$,
$$G^{\star}:=\{x\in\mathbb{R}^n:\;\omega_n|x|^n<|G|\}.$$

For $f\in  W_0^{1,\Phi}(G)$, we define the {\it symmetric decreasing rearrangement} $f^{\star}$ of $f$ by symmetrizing its level sets, that is
\begin{eqnarray}
f^{\star}(x)=\int_{0}^{\infty}\mathcal{X}_{[f]^{\star}_h}(x)dh,\;\;\;x\in G^{\star}.
\end{eqnarray}

 For Sobolev spaces $W^{1,p}(\mathbb{R}^n)$, Burchard \cite[Proposition 7.1]{Bu97} proved the following proposition on the approximation of spherical symmetrization by Steiner symmetrizations.
 \begin{pro}\label{19f}  {\rm (Convergence of the $W^{1,1}$-norm)}. Let $f$ be a nonnegative function in $W^{1,p}(\mathbb{R}^n)$, $n\geq2$ and $p\geq 1$, that vanishes at infinity. There exists a sequence of successive Steiner symmetrizations $\{f_k\}_{k\geq 0}$ of $f$  so that
 $$f_k\rightharpoonup f^{\star}\;{\rm weakly\;in}\;W^{1,1}.$$
 \end{pro}
\begin{rem}\label{19b}
 For Orlicz-Sobolev spaces $W^{1,\Phi}(\mathbb{R}^n)$, there does not exist a result similar to that of Proposition \ref{19f}. Thus we consider the problem in $W^{1,p}(\mathbb{R}^n)$.  Since $f\in W_0^{1,1}(G)$ for $f\in W_0^{1,\Phi}(G)$, there exists a sequence of successive Steiner symmetrizations $\{\bar{f}_k\}_{k\geq 0}$ of $\bar{f}$ so that
$$\bar{f}_k\rightharpoonup \bar{f}^{\star}\;{\rm weakly\;in}\;W^{1,1}(\mathbb{R}^n).$$
\end{rem}

\section{Definition and basic properties of Orlicz-Sobolev affine balls}\label{s4}

The Orlicz-Sobolev affine ball $B_{\phi}(f)$ of $f\in W_0^{1,\Phi}(G)$ is defined as the unit ball  of the $n$-dimensional Banach space whose
norm is given by
\begin{eqnarray}\label{48}
\|y\|_{f,\phi}=\inf\left\{\lambda> 0:\;\frac{1}{|G|}\int_{G}\phi\left(\frac{y\cdot\nabla f(x)}{\lambda}\right)dx\leq 1\right\},\;y\in\mathbb{R}^n.
\end{eqnarray}
And the volume of the Orlicz-Sobolev affine ball is given by
\begin{eqnarray}\label{27}
|B_{\phi}(f)|=\frac{1}{n}\int_{S^{n-1}}\|v\|_{f,\phi}^{-n}dv,
\end{eqnarray}
where $dv$ denotes the spherical Lebesgue measure.

Since $f\in W_0^{1,\Phi}(G)$, it is impossible that there exists some $u_0\in S^{n-1}$ such that $\nabla f(x)\cdot u_0\geq0$  for almost all $x\in G$. Since $\phi$ is strictly increasing on $[0,\infty)$ or strictly decreasing on $(-\infty,0]$ it follows that for $y\neq0$ the function
$$\lambda\mapsto\frac{1}{|G|}\int_{G}\phi\left(\frac{y\cdot\nabla f(x)}{\lambda}\right)dx$$
is strictly decreasing in $(0,\infty)$. Thus, we have the following lemma.
\begin{lem}\label{47}
If $f\in W_0^{1,\Phi}(G)$ and $y_0\in\mathbb{R}^n\backslash \{0\}$, then
$$\frac{1}{|G|}\int_{G}\phi\left(\frac{y_0\cdot\nabla f(x)}{\lambda_0}\right)dx=1$$
if and only if
$$\|y_0\|_{f,\phi}=\lambda_0.$$
\end{lem}

  The following Lemma \ref{9c}, Lemma \ref{2a} and Lemma \ref{57}  demonstrate the affine invariance of $\mathcal{E}_{\phi}(f)$, non-negativity  and boundedness of $\|\cdot\|_{f,\phi}$, respectively.
\begin{lem}\label{9c} {\rm (\cite[Lemma 4.2]{LYJ17})}.
If  $f\in W_0^{1,\Phi}(G)$, then
$\mathcal{E}_{\phi}(f)$
is invariant under $SL(n)$ transformations.
\end{lem}

\begin{lem}\label{2a} {\rm (\cite[Lemma 4.3]{LYJ17})}.
If $f\in W_0^{1,\Phi}(G)$, then $\|\cdot\|_{f,\phi}$ defines a norm on the Banach space $(\mathbb{R}^n,\|\cdot\|_{f,\phi})$. In particular, $\|v\|_{f,\phi}>0$ for any $v\in S^{n-1}$.
\end{lem}

\begin{lem}\label{57} {\rm(\cite[Lemma 4.4]{LYJ17})}.
If $f\in W_0^{1,\Phi}(G)$ and let
\begin{eqnarray}\label{56}
c_{\phi}=\max\left\{c>0:\max\{\phi(c),\phi(-c)\}\leq 1\right\},
\end{eqnarray}
then for any $v\in S^{n-1}$, we have
\begin{eqnarray}
\frac{\int_Gf(x)dx}{c_{\phi}|G|D_G}\leq \|v\|_{f,\phi}\leq \frac{\sup\{|\nabla f(x)|:x\in G\}}{c_{\phi}}.
\end{eqnarray}
\end{lem}

The following lemma shows that the Orlicz-Sobolev affine ball operator $B_{\phi}: W_0^{1,\Phi}(G)\rightarrow \mathcal{K}_o^n$ is weakly continuous in some sense.
\begin{lem}\label{72}
For $f_i\in W_0^{1,\Phi}(G_i)$, $i=0,1,2,\dots$, if
\begin{eqnarray}\label{6h}
\bar{f}_i\rightharpoonup\bar{f}_0, {\rm weakly\;in}\;W^{1,1}(\mathbb{R}^n),
\end{eqnarray}
then there exists a subsequence of  $\{B_{\phi}(f_i)\}_{i=1}^{\infty}$, denoted by   $\{B_{\phi}(f_i)\}_{i=1}^{\infty}$ as well, and a convex body $K_0$ such that $o\in K_0$,
\begin{eqnarray}\label{6c}
\lim_{i\rightarrow \infty}\delta(B_{\phi}(f_i),K_0)=0
\end{eqnarray}
and
\begin{eqnarray}
K_0\subset B_{\phi}(f_0).
\end{eqnarray}
\end{lem}

\begin{proof}
For $u_0\in S^{n-1}$, let
\begin{eqnarray}\label{6g}
\|u_0\|_{f_i,\phi}=\lambda_i,
\end{eqnarray}
and note that Lemma \ref{57} gives
\begin{eqnarray}\label{6e}
0<\frac{\int_{G_i}f_i(x)dx}{c_{\phi}|G_i|D_{G_i}}\leq\lambda_i.
\end{eqnarray}
Moreover, by (\ref{6h}), we have
\begin{eqnarray}\label{6f}
\lim_{i\rightarrow\infty}\frac{\int_{G_i}f_i(x)dx}{c_{\phi}|G_i|D_{G_i}}=\frac{\int_{G_0}f_0(x)dx}{c_{\phi}|G_0|D_{G_0}}>0,
\end{eqnarray}
which implies that there exists a real number $m>0$ such that $\|u\|_{f_i,\phi}>m$ for any $u\in S^{n-1}$ and any positive integer $i$. Thus the radial functions  $\rho(B_{\phi}(f_i),u)<\frac{1}{m}$ for any $u\in S^{n-1}$ and any positive integer $i$. By Blaschke selection theorem
(see \cite[Theorem 1.8.7]{Schneider13}), there exists a subsequence of $\{B_{\phi}(f_i)\}_{i=1}^{\infty}$,  denoted by   $\{B_{\phi}(f_i)\}_{i=1}^{\infty}$ as well, that converges to a convex body $K_0\in\mathcal{K}^n$.
By Lemma \ref{2a}, $\|u\|_{f_i,\phi}>0$ for any $u\in S^{n-1}$ and any positive integer $i$. Thus $o\in K_0$.

Let $\|u_0\|_{K_0}=\lambda_{\ast}$. By (\ref{6c}) and (\ref{6g}), we have
\begin{eqnarray}
\lim_{i\rightarrow \infty}\lambda_i=\lambda_{\ast}.
\end{eqnarray}

Let $\tilde{f_i}=\bar{f_i}/\lambda_i$. Since $\lambda_i \rightarrow\lambda_{\ast}$ and $\bar{f}_i\rightharpoonup\bar{f}_0$  weakly in  $W^{1,1}(\mathbb{R}^n)$, we have
\begin{eqnarray}\label{7d}
\tilde{f}_i\rightharpoonup \bar{f}_0/\lambda_{\ast}\;\;{\rm weakly\;in}\;\;W^{1,1}(\mathbb{R}^n).
\end{eqnarray}

The fact that $\|u_0\|_{f_i,\phi}=\lambda_i$, together with Lemma \ref{47} and (\ref{3a}), shows that
\begin{eqnarray}\label{7f}
\frac{1}{|G_i|}\int_{\mathbb{R}^n}\phi\left(u_0\cdot \nabla\tilde{f}_i(x)\right)dx=1\;{\rm for\;all}\;i.
\end{eqnarray}

Since $\phi$ is a convex function, by \cite[Theorem 1 in P.19]{Evans90}, the convex gradient integral
$$\frac{1}{|G|}\int_{\mathbb{R}^n}\phi\left(u_0\cdot \nabla\bar{f}(x)\right)dx$$
is lower semicontinuous with respect to weak convergence in $W^{1,1}(\mathbb{R}^n)$. By (\ref{7d}) and (\ref{7f}), we have
$$\frac{1}{|G_0|}\int_{\mathbb{R}^n}\phi\left(\frac{u_0\cdot \nabla\bar{f}_0(x)}{\lambda_{\ast}}\right)dx\leq\lim_{i\rightarrow \infty}\frac{1}{|G_i|}\int_{\mathbb{R}^n}\phi\left(u_0\cdot \nabla\tilde{f}_i(x)\right)dx= 1.$$

This, together with (\ref{3a}) and  the definition (\ref{48}) yields
\begin{eqnarray}\label{3b}
\|u_0\|_{f_0,\phi}\leq\lambda_{\ast}=\|u_0\|_{K_0}.
\end{eqnarray}
By (\ref{3b}) and the arbitrariness of $u_0\in S^{n-1}$,  we have $K_0\subset B_{\phi}(f_0)$.
\end{proof}

\section{Proof of the main theorem}\label{s5}

\begin{lem}\label{17}
 If $f\in W_0^{1,\Phi}(G)$ and $x^{\prime}\in{\rm int}[f]^{\prime}_{h,i}$, then
\begin{eqnarray}\label{18}
\frac{\left|\left(-\nabla\overline{\varrho}_{h,i}(x^{\prime}),1\right)\right|}{\left|\nabla Sf\left(x^{\prime},\overline{\varrho}_{h,i}(x^{\prime})\right)\right|}
=\frac{1}{2}\sum_{k=1}^{i}\frac{\left|(-\nabla\underline{\ell}^k_{h,i}(x^{\prime}),-1)\right|}
{\left|\nabla f\left(x^{\prime},-\underline{\ell}^k_{h,i}(x^{\prime})\right)\right|}+\frac{1}{2}\sum_{k=1}^{i}\frac{\left|\left(-\nabla\overline{\ell}_{h,i}^k(x^{\prime}),1\right)\right|}{\left|\nabla f\left(x^{\prime},\overline{\ell}^k_{h,i}(x^{\prime})\right)\right|}.\end{eqnarray}
\end{lem}
\begin{proof}
In fact, the vectors $\left(-\nabla\overline{\varrho}_{h,i}(x^{\prime}),1\right)$ and $\nabla Sf\left(x^{\prime},\overline{\varrho}_{h,i}(x^{\prime})\right)$ have the same direction, i.e., the direction of the exterior  normal vector of $[Sf]_h$ at the point $(x^{\prime},\overline{\varrho}_{h,i}(x^{\prime}))\in \partial^{\ast} [Sf]_h$. Thus for the left side of (\ref{18}) we have
\begin{eqnarray}\label{13a}
\frac{\left|\left(-\nabla\overline{\varrho}_{h,i}(x^{\prime}),1\right)\right|}{\left|\nabla Sf\left(x^{\prime},\overline{\varrho}_{h,i}(x^{\prime})\right)\right|}=\frac{1}{\left|\frac{\partial Sf}{\partial x_n}(x^{\prime},\overline{\varrho}_{h,i}(x^{\prime}))\right|}.
\end{eqnarray}

 Similarly, the vectors $\left(-\nabla\overline{\ell}_{h,i}^k(x^{\prime}),1\right)$ and $\nabla f\left(x^{\prime},\overline{\ell}^k_{h,i}(x^{\prime})\right)$ have the same direction, i.e., the direction of the exterior  normal vector to $[f]_h$ at the point $(x^{\prime},\overline{\ell}^k_{h,i}(x^{\prime}))\in \partial^{\ast}[f]_h$, we have
\begin{eqnarray}\label{18b}
\frac{\left|\left(-\nabla\overline{\ell}_{h,i}^k(x^{\prime}),1\right)\right|}{\left|\nabla f\left(x^{\prime},\overline{\ell}^k_{h,i}(x^{\prime})\right)\right|}=\frac{1}{\left|\frac{\partial f}{\partial x_n}\left(x^{\prime},\overline{\ell}^k_{h,i}(x^{\prime})\right)\right|}.
\end{eqnarray}

Moreover, we have
\begin{eqnarray}\label{18c}
\frac{\left|(-\nabla\underline{\ell}^k_{h,i}(x^{\prime}),-1)\right|}{\left|\nabla f\left(x^{\prime},-\underline{\ell}^k_{h,i}(x^{\prime})\right)\right|}=\frac{1}{\left|\frac{\partial f}{\partial x_n}\left(x^{\prime},-\underline{\ell}^k_{h,i}(x^{\prime})\right)\right|}.
\end{eqnarray}

By (\ref{13a}), (\ref{18b}) and (\ref{18c}), (\ref{18}) is equivalent to
\begin{eqnarray}\label{18d}
\frac{1}{\left|\frac{\partial Sf}{\partial x_n}(x^{\prime},\overline{\varrho}_{h,i}(x^{\prime}))\right|}=\frac{1}{2}\sum_{k=1}^{i}\frac{1}{\left|\frac{\partial f}{\partial x_n}\left(x^{\prime},-\underline{\ell}^k_{h,i}(x^{\prime})\right)\right|}+\frac{1}{2}\sum_{k=1}^{i}\frac{1}{\left|\frac{\partial f}{\partial x_n}\left(x^{\prime},\overline{\ell}^k_{h,i}(x^{\prime})\right)\right|}.
\end{eqnarray}
By \cite[Lemma 4.1]{CF06} (also see  \cite[(5.1)]{Bu97}), (\ref{18d}) is established.
\end{proof}

Since $Sf$ is symmetric with respect to $e_n^{\perp}$, i.e., $Sf(x^{\prime}+re_n)=Sf(x^{\prime}-re_n)$ for any $x^{\prime}+re_n\in SG$, it is easy to check that for $x^{\prime}\in {\rm int}[f]^{\prime}_{h,i}$,
\begin{eqnarray}\label{13}
\frac{\left|\left(-\nabla\overline{\varrho}_{h,i}(x^{\prime}),1\right)\right|}{\left|\nabla Sf\left(x^{\prime},\overline{\varrho}_{h,i}(x^{\prime})\right)\right|}&=& \frac{\left|\left(-\nabla\underline{\varrho}_{h,i}(x^{\prime}),-1\right)\right|}{\left|\nabla Sf\left(x^{\prime},-\underline{\varrho}_{h,i}(x^{\prime})\right)\right|}. \end{eqnarray}

\begin{lem}\label{1c}
If $f\in W_0^{1,\Phi}(G)$ and let
$$G_1=\left\{x\in G:\;\frac{\partial f}{\partial x_n}(x)=0\right\}$$
and
$$G_2=\left\{x\in SG:\;\frac{\partial Sf}{\partial x_n}(x)=0\right\},$$ then for any $z\in \mathbb{R}^n$, we have
\begin{eqnarray}\label{6a}
\int_{G_1}\phi(z\cdot\nabla f(x))dx=\int_{G_2}\phi(z\cdot\nabla Sf(x))dx.
\end{eqnarray}
\end{lem}
\begin{proof}
By \cite[Proposition 2.3]{CF06}, if $f\in W_0^{1,1}(G)$, then for $\mathcal{L}^{n-1}$-a.e. $x^{\prime}\in G^{\prime}$,
\begin{eqnarray}\label{6j}
|\{x_n:(x^{\prime},x_n)\in G,\;\frac{\partial f}{\partial x_n}(x^{\prime},x_n)=0\}|=|\{x_n:(x^{\prime},x_n)\in SG,\;\frac{\partial Sf}{\partial x_n}(x^{\prime},x_n)=0\}|.
\end{eqnarray}
By (\ref{6j}), if $\mathcal{L}^n(G_1)=0$, then $\mathcal{L}^n(G_2)=0$. Thus (\ref{6a}) is established.

If $\mathcal{L}^n(G_1)>0$, then by (\ref{6j})  $\mathcal{L}^n(G_2)=\mathcal{L}^n(G_1)>0$. Moreover, by (\ref{6j}), the essential projections of $G_1$ and $G_2$ onto $e_n^{\perp}$ satisfy $G_1^{\prime,+}=G_2^{\prime,+}$. Next, we prove that for $\mathcal{L}^{n-1}$-a.e. $x^{\prime}\in G_1^{\prime,+}$,
\begin{eqnarray}\label{7g}
\int_{\{x_n:\;(x^{\prime},x_n)\in G_1\}}\phi\left(z\cdot\nabla f(x^{\prime},x_n)\right)dx_n=\int_{\{x_n:\;(x^{\prime},x_n)\in G_2\}}\phi\left(z\cdot\nabla Sf(x^{\prime},x_n)\right)dx_n.
\end{eqnarray}

For fixed $x^{\prime}\in G_1^{\prime,+}$, let
\begin{eqnarray}
D_1=\{h\in\mathbb{R}:\;h=f(x^{\prime},x_n)\;{\rm and}\;(x^{\prime},x_n)\in G_1\}
\end{eqnarray}
and
\begin{eqnarray}
D_2=\{h\in\mathbb{R}:\;h=Sf(x^{\prime},x_n)\;{\rm and}\;(x^{\prime},x_n)\in G_2\}.
\end{eqnarray}

Let $f^{x^{\prime}}(x_n):=f(x^{\prime},x_n)$ and $(Sf)^{x^{\prime}}(x_n):=Sf(x^{\prime},x_n)$. Since ${\rm sub}((Sf)^{x^{\prime}})$ is equivalent to $S({\rm sub}f^{x^{\prime}})$, we have $D_1=D_2$. For $h\in D_1$, let $$D^h_{x^{\prime}}=\{x_n:\;(x^{\prime},x_n)\in G_1,\;f(x^{\prime},x_n)=h\}$$ and
$$\tilde{D}^h_{x^{\prime}}=\{x_n:\;(x^{\prime},x_n)\in G_2,\;Sf(x^{\prime},x_n)=h\}.$$
Next, we prove that for any $h\in D_1$,
\begin{eqnarray}\label{7a}
\int_{D^h_{x^{\prime}}}\phi\left(z\cdot\nabla f(x^{\prime},x_n)\right)dx_n=\int_{\tilde{D}^h_{x^{\prime}}}\phi\left(z\cdot\nabla Sf(x^{\prime},x_n)\right)dx_n.
\end{eqnarray}

By \cite[Lemma 4.3]{Bu97}, for $\mathcal{L}^1$-a.e. $x_n\in D_{x^{\prime}}^{h}$, $\nabla f(x^{\prime},x_n)$ is equal to a constant. By \cite[Lemma 5.1]{Bu97}, for $\mathcal{L}^1$-a.e. $x_n\in \tilde{D}_{x^{\prime}}^{h}$, $\nabla Sf(x^{\prime},x_n)$ equals the same constant as $\nabla f(x^{\prime},x_n)$ for $x_n\in D_{x^{\prime}}^{h}$. Thus by $|D^h_{x^{\prime}}|=|\tilde{D}^h_{x^{\prime}}|$, (\ref{7a}) is established. By (\ref{7a}) and the arbitrariness of $h\in D_1$, (\ref{7g}) is established. By (\ref{7g}) and Fubini's theorem, (\ref{6a}) is established.
\end{proof}

\begin{pro}\label{2}
If $f\in  W_0^{1,\Phi}(G)$, then
\begin{eqnarray}\label{65}
S(B_{\phi}(f))\subset B_{\phi}(Sf).
\end{eqnarray}
\end{pro}
\begin{proof}
 Let
$$\|(y^{\prime},t)\|_{f,\phi}=1\;\;{\rm and}\;\;\|(y^{\prime},-s)\|_{f,\phi}=1,$$
with $t\neq-s$.
By Lemma \ref{47}, this means that
\begin{eqnarray}\label{62}
\frac{1}{|G|}\int_{G}\phi\left((y^{\prime},t)\cdot \nabla f(x)\right)dx=1
\end{eqnarray}
and
\begin{eqnarray}\label{63}
\frac{1}{|G|}\int_{G}\phi\left((y^{\prime},-s)\cdot \nabla f(x)\right)dx=1.
\end{eqnarray}

By Lemma \ref{46}, the desired inclusion will been established if we can show that
\begin{eqnarray}\label{64}
\|(y^{\prime},t/2+s/2)\|_{Sf,\phi}\leq 1.
\end{eqnarray}

For $x^{\prime}\in{\rm int}[f]^{\prime}_{h,i}$ and $k=1,2,\cdots,i$, let
\begin{eqnarray}\label{10b}
c^k_{h,i}=\frac{|(-\nabla\overline{\ell}^k_{h,i}(x^{\prime}),1)|}{|\nabla f(x^{\prime},\overline{\ell}^k_{h,i}(x^{\prime}))|}\;\;{\rm and}\;\;d^k_{h,i}=\frac{|(-\nabla\underline{\ell}^k_{h,i}(x^{\prime}),-1)|}{|\nabla f(x^{\prime},-\underline{\ell}^k_{h,i}(x^{\prime}))|}.
\end{eqnarray}

For $x^{\prime}\in {\rm int}[f]^{\prime}_{h,i}$, let
\begin{eqnarray}
a_1=\frac{|(-\nabla\overline{\varrho}_{h,i}(x^{\prime}),1)|}{|\nabla Sf(x^{\prime},\overline{\varrho}_{h,i}(x^{\prime}))|}\;\;{\rm and}\;\;
a_2=\frac{|(-\nabla\underline{\varrho}_{h,i}(x^{\prime}),-1)|}{|\nabla Sf(x^{\prime},-\underline{\varrho}_{h,i}(x^{\prime}))|}.
\end{eqnarray}
By (\ref{13}) and Lemma \ref{17},
\begin{eqnarray}\label{7b}
a_1=a_2=\frac{1}{2}\sum_{k=1}^{i}c^k_{h,i}+\frac{1}{2}\sum_{k=1}^{i}d^k_{h,i}.
\end{eqnarray}
Since $\nabla Sf(x^{\prime},\overline{\varrho}_{h,i}(x^{\prime}))$ and $(-\nabla\overline{\varrho}_{h,i}(x^{\prime}),1)$ have the same directions,
\begin{eqnarray}\label{8b}
\nabla Sf(x^{\prime},\overline{\varrho}_{h,i}(x^{\prime}))=\frac{(-\nabla\overline{\varrho}_{h,i}(x^{\prime}),1)}{a_1}.
\end{eqnarray}
Similarly, we have
\begin{eqnarray}\label{9b}
\nabla Sf(x^{\prime},-\underline{\varrho}_{h,i}(x^{\prime}))=\frac{(-\nabla\underline{\varrho}_{h,i}(x^{\prime}),-1)}{a_2},
\end{eqnarray}
\begin{eqnarray}\label{10a}
\nabla f(x^{\prime},\overline{\ell}^k_{h,i}(x^{\prime}))=\frac{(-\nabla\overline{\ell}^k_{h,i}(x^{\prime}),1)}{c^k_{h,i}},\;\;k=1,2,\dots,i
\end{eqnarray}
and
\begin{eqnarray}\label{10c}
\nabla f(x^{\prime},-\underline{\ell}^k_{h,i}(x^{\prime}))=\frac{(-\nabla\underline{\ell}^k_{h,i}(x^{\prime}),-1)}{d^k_{h,i}},\;\;k=1,2,\dots,i.
\end{eqnarray}

By (\ref{53}), (\ref{8b}), (\ref{9b}), (\ref{a6}), Lemma \ref{88} and (\ref{7b}), we have that

\begin{eqnarray}\label{54}
&&\int_{\partial^{\ast}[Sf]_h/G_2}\phi\left((y^{\prime},t/2+s/2)\cdot \nabla Sf(x)\right)|\nabla Sf(x)|^{-1}d\mathcal {H}^{n-1}(x)\nonumber\\
&=&\int_{([f]_h/G_1)^{\prime}}\phi\left((y^{\prime},t/2+s/2)\cdot \nabla Sf(x^{\prime},\overline{\varrho}_h(x^{\prime}))\right)a_1 dx^{\prime}\nonumber\\
&&+\int_{([f]_h/G_1)^{\prime}}\phi\left((y^{\prime},t/2+s/2)\cdot \nabla Sf(x^{\prime},-\underline{\varrho}_h(x^{\prime}))\right)a_2dx^{\prime}\nonumber\\
&=&\sum_{i=1}^{\infty}\int_{{\rm int}[f]^{\prime}_{h,i}}\phi\left(\frac{1}{2a_1}\left(t+s-y^{\prime}\cdot\sum_{k=1}^{i}\nabla(\overline{\ell}^k_{h,i}+\underline{\ell}^k_{h,i})(x^{\prime})\right)\right)a_1 dx^{\prime}\nonumber\\
&&+\sum_{i=1}^{\infty}\int_{{\rm int}[f]^{\prime}_{h,i}}\phi\left(\frac{1}{2a_2}\left(-t-s-y^{\prime}\cdot\sum_{k=1}^{i}\nabla(\overline{\ell}^k_{h,i}+\underline{\ell}^k_{h,i})(x^{\prime})\right)\right)a_2dx^{\prime}\nonumber\\
&\leq&\sum_{i=1}^{\infty}\int_{{\rm int}[f]^{\prime}_{h,i}}\phi\left(\frac{1}{2a_1}\left(i(t+s)-y^{\prime}\cdot\sum_{k=1}^{i}\nabla(\overline{\ell}^k_{h,i}+\underline{\ell}^k_{h,i})(x^{\prime})\right)\right)a_1 dx^{\prime}\nonumber\\
&&+\sum_{i=1}^{\infty}\int_{{\rm int}[f]^{\prime}_{h,i}}\phi\left(\frac{1}{2a_2}\left(i(-t-s)-y^{\prime}\cdot\sum_{k=1}^{i}\nabla(\overline{\ell}^k_{h,i}+\underline{\ell}^k_{h,i})(x^{\prime})\right)\right)a_2dx^{\prime}\nonumber\\
&=&\frac{1}{2}\sum_{i=1}^{\infty}\int_{{\rm int}[f]^{\prime}_{h,i}}\phi\left(\frac{\sum_{k=1}^{i}(t-y^{\prime}\cdot\nabla\overline{\ell}^k_{h,i}(x^{\prime}))+\sum_{k=1}^{i}(s-y^{\prime}\cdot\nabla\underline{\ell}^k_{h,i}(x^{\prime}))}{\sum_{k=1}^{i}(c^k_{h,i}+d^k_{h,i})}\right)\sum_{k=1}^{i}(c^k_{h,i}+d^k_{h,i})dx^{\prime}\nonumber\\
&&+\frac{1}{2}\sum_{i=1}^{\infty}\int_{{\rm int}[f]^{\prime}_{h,i}}\phi\left(\frac{\sum_{k=1}^{i}(-t-y^{\prime}\cdot\nabla\underline{\ell}^k_{h,i}(x^{\prime}))+\sum_{k=1}^{i}(-s-y^{\prime}\cdot\nabla\overline{\ell}^k_{h,i}(x^{\prime}))}{\sum_{k=1}^{i}(d^k_{h,i}+c^k_{h,i})}\right)\sum_{k=1}^{i}(d^k_{h,i}+c^k_{h,i})dx^{\prime}\nonumber\\
\end{eqnarray}
By (\ref{a7}), (\ref{10a}), (\ref{10c}) and (\ref{53}) again, the last expression
\begin{eqnarray}
&\leq&\frac{1}{2}\sum_{i=1}^{\infty}\int_{{\rm int}[f]^{\prime}_{h,i}}\sum_{k=1}^{i}\phi\left((y^{\prime},t)\cdot \nabla f(x^{\prime},\overline{\ell}^k_{h,i}(x^{\prime}))\right)c^k_{h,i}dx^{\prime}\nonumber\\
&&+\frac{1}{2}\sum_{i=1}^{\infty}\int_{{\rm int}[f]^{\prime}_{h,i}}\sum_{k=1}^{i}\phi\left((y^{\prime},-s)\cdot \nabla f(x^{\prime},-\underline{\ell}^k_{h,i}(x^{\prime}))\right)d^k_{h,i}dx^{\prime}\nonumber\\
&&+\frac{1}{2}\sum_{i=1}^{\infty}\int_{{\rm int}[f]^{\prime}_{h,i}}\sum_{k=1}^{i}\phi\left((y^{\prime},t)\cdot \nabla f(x^{\prime},-\underline{\ell}^k_{h,i}(x^{\prime}))\right)d^k_{h,i}dx^{\prime}\nonumber\\
&&+\frac{1}{2}\sum_{i=1}^{\infty}\int_{{\rm int}[f]^{\prime}_{h,i}}\sum_{k=1}^{i}\phi\left((y^{\prime},-s)\cdot \nabla f(x^{\prime},\overline{\ell}^k_{h,i}(x^{\prime}))\right)c^k_{h,i}dx^{\prime}\nonumber\\
&=&\frac{1}{2}\int_{\partial^{\ast}[f]_h/G_1}\phi\left((y^{\prime},t)\cdot \nabla f(x)\right)|\nabla f(x)|^{-1}d\mathcal {H}^{n-1}(x)\nonumber\\
&~&+\frac{1}{2}\int_{\partial^{\ast}[f]_h/G_1}\phi\left((y^{\prime},-s)\cdot \nabla f(x)\right)|\nabla f(x)|^{-1}d\mathcal {H}^{n-1}(x).
\end{eqnarray}
Integrating both sides of the above inequality on $[0,\infty)$ with respect to $h$, it follows from the co-area formula (\ref{1e}), Lemma \ref{1c}, (\ref{62}) and (\ref{63}),  that
\begin{eqnarray}
&&\int_{SG}\phi\left((y^{\prime},t/2+s/2)\cdot \nabla Sf(x)\right)dx\nonumber\\
&\leq&\frac{1}{2}\int_{G}\phi\left((y^{\prime},t)\cdot \nabla f(x)\right)dx+\frac{1}{2}\int_{G}\phi\left((y^{\prime},-s)\cdot \nabla f(x)\right)dx\nonumber\\
&=&1.
\end{eqnarray}
This and a glance at definition (\ref{48}), gives (\ref{64}), and thus (\ref{65}) is proved.
\end{proof}

{\bf  \noindent Proof of Theorem \ref{1}.} By Theorem \ref{19f}, there exists a sequence of  directions $\{u_i\}$ such that the sequence defined by $f_{i+1}=S_{u_i}f_i$ (where $i=0,1,\dots$ and $f_0=f$) converges to $f^{\star}$ weakly in $W^{1,1}$. By Lemma \ref{72}, there exists a convex body $K_0$ such that $B_{\phi}(f_i)$ converges to $K_0$ and $K_0\subset B_{\phi}(f^{\star})$. Thus by Proposition \ref{2},
$$|B_{\phi}(f)|\leq|B_{\phi}(f_1)|\leq\cdots\leq|B_{\phi}(f_i)|\rightarrow |K_0|\leq |B_{\phi}(f^{\star})|.$$
Thus, we have
$$\mathcal{E}_{\phi}(f^{\star})\leq\mathcal{E}_{\phi}(f).$$ \qed

\section{The affine Orlicz  P\'olya-Szeg\" o principle and the Orlicz-Petty projection inequality for star bodies}\label{s6}

In \cite{LYZ10}, the {\it Orlicz projection body $\Pi_{\phi}K$} of $K\in\mathcal{K}^n_o$ is defined as the body whose support function is given by
\begin{eqnarray}
h_{\Pi_{\phi}K}(u)=\inf\left\{\lambda>0:\int_{\partial K}\phi\left(\frac{u\cdot v_K(y)}{\lambda h_K(v_K(y))}\right)h_K(v_K(y))d\mathcal{H}^{n-1}(y)\leq n|K|\right\}.
\end{eqnarray}
Lutwak, Yang and Zhang \cite{LYZ10}  proved the following remarkable inequality.

{\bf \noindent Orlicz-Petty Projection inequality.} Suppose $\phi\in\mathcal{N}$. If $K\in\mathcal{K}^n_o$, then the volume ratio
$$|\Pi^{\ast}_{\phi}K|/|K|$$
is maximized when $K$ is an ellipsoid centered at the origin.

Zhang \cite{Zhang99} gave the definition of the projection body for a compact set with piecewise $C^1$. Similarly, for $K\in\mathcal{S}_o^n\cap \mathcal{C}^n$, we define the {\it Orlicz projection body $\Pi_{\phi}K$} of $K$ as
\begin{eqnarray}\label{19a}
h_{\Pi_{\phi}K}(u)=\inf\left\{\lambda>0:\int_{\partial^{\ast} K}\phi\left(\frac{u\cdot v_K(y)}{\lambda y\cdot v_K(y)}\right)y\cdot v_K(y)d\mathcal{H}^{n-1}(y)\leq n|K|\right\}.
\end{eqnarray}
It is easily seen that $\Pi_{\phi}K$ is a convex body  containing the origin in its interior.

In this section, we shall prove that the affine Orlicz  P\'olya-Szeg\" o principle (\ref{a1}) implies the following Orlicz-Petty projection inequality for star bodies.

{\bf Orlicz-Petty projection inequality for star bodies.} Suppose $\phi\in\mathcal{N}$. If $K\in\mathcal{S}^n_o\cap\mathcal{C}^n$, then the volume ratio
$$|\Pi^{\ast}_{\phi}K|/|K|$$
is maximized when $K$ is an ellipsoid centered at the origin.

In (\ref{a1}), let
\begin{eqnarray}\label{7c}
f(x)=1-\|x\|_K,\;\;\; x\in {\rm int} K,
\end{eqnarray}
where $K\in \mathcal{S}_o^n\cap\mathcal{C}^n$. Since $[f]_h=(1-h)K$ is a set of finite perimeter for any $h>0$ and $\int_{{\rm int}K}|\nabla f|dx=\int_{0}^{1}P([f]_h)dh<\infty$, $f\in W_0^{1,1}({\rm int}K)$.
 It is easy to check that
\begin{eqnarray}\label{1i}
f^{\star}(x)=1-\|x\|_{K^{\star}},\;\;\; x\in {\rm int} K^{\star}.
\end{eqnarray}
For $f$ as in (\ref{7c}), since $\phi(0)=0$ and $0<f\leq 1$, by (\ref{1e}), (\ref{2h}) and (\ref{2i}), we have
\begin{eqnarray}\label{28}
&&\frac{1}{|K|}\int_{{\rm int K}}\phi\left(\frac{u\cdot\nabla f(x)}{\lambda}\right)dx\nonumber\\
&=&\frac{1}{|K|}\int_{\{x\in {\rm int}K,\nabla f(x)\neq 0\}}\phi\left(\frac{u\cdot\nabla f(x)}{\lambda}\right)dx\nonumber\\
&=&\frac{1}{|K|}\int^{1}_{0}\int_{(1-h)\partial^{\ast} K}\phi\left(\frac{u\cdot\nabla f(x)}{\lambda}\right)|\nabla f(x)|^{-1}d\mathcal{H}^{n-1}(x)dh\nonumber\\
&=&\frac{1}{|K|}\int^{1}_{0}\int_{\partial^{\ast} K}\phi\left(\frac{-u\cdot v_K(y)}{\lambda y\cdot v_K(y)}\right)y\cdot v_K(y)(1-h)^{n-1}d\mathcal{H}^{n-1}(y)dh\nonumber\\
&=&\frac{1}{n|K|}\int_{\partial^{\ast} K}\phi\left(\frac{-u\cdot v_K(y)}{\lambda y\cdot v_K(y)}\right)y\cdot v_K(y)d\mathcal{H}^{n-1}(y).
\end{eqnarray}

By (\ref{48}), (\ref{19a}) and (\ref{28}), we have
\begin{eqnarray}\label{28a}
\|u\|_{f,\phi}&=&\inf\left\{\lambda>0:\frac{1}{|K|}\int_{{\rm int}K}\phi\left(\frac{u\cdot\nabla f(x)}{\lambda}\right)dx\leq1\right\}\nonumber\\
&=&\inf\left\{\lambda>0:\int_{\partial^{\ast} K}\phi\left(\frac{-u\cdot v(y)}{\lambda y\cdot v(y)}\right)y\cdot v(y)d\mathcal{H}^{n-1}(y)\leq n|K|\right\}\nonumber\\
&=&h_{\Pi_{\phi}K}(-u).
\end{eqnarray}
By (\ref{41a}), (\ref{28a}) and the definition of $B_{\phi}(f)$, we have
\begin{eqnarray}\label{34}
B_{\phi}(f)=-\Pi^{\ast}_{\phi}K.
\end{eqnarray}
By (\ref{34}), $|K^{\ast}|=|K|$, $\mathcal{E}_{\phi}(f^{\star})\leq \mathcal{E}_{\phi}(f)$ and $\mathcal{E}_{\phi}(f)=|B_{\phi}(f)|^{-1/n}$, we have
\begin{eqnarray}\label{7e}
\frac{|\Pi^{\ast}_{\phi}K|}{|K|}=\frac{|B_{\phi}(f)|}{|K|}\leq \frac{|B_{\phi}(f^{\star})|}{|K^{\star}|}=\frac{|\Pi^{\ast}_{\phi}K^{\star}|}{|K^{\star}|},
\end{eqnarray}
which is the Orlicz projection inequality for star bodies.

\section{The Orlicz and affine Orlicz P\'olya-Szeg\" o inequalities}\label{s7}

In this section, we shall prove that the affine Orlicz P\'olya-Szeg\" o inequality (\ref{a1}) is essentially stronger than the Euclidean Orlicz P\'olya-Szeg\" o inequality (\ref{1d}). Throughout this section, let $\phi\in\mathcal{N}$  be an even function.\\
\begin{lem}\label{7.1}
For $f\in  W_0^{n,\Phi}(G)$, then
\begin{eqnarray}
\frac{2\omega_{n-1}}{n\omega_n^{(n+1)/n}}\leq \frac{\mathcal {E}_{\phi}(f^{\star})}{\|\nabla f^{\star}\|_{\Phi}}<\omega_n^{-\frac{1}{n}}.
\end{eqnarray}

\end{lem}
\begin{proof} For $v\in S^{n-1}$, since $\phi\in\mathcal{N}$ is an even function and  by Jensen's inequality, we have
\begin{eqnarray}\label{66}
\|v\|_{f^{\star},\phi}&=&\inf\left\{\lambda> 0:\int_{G^{\ast}}\phi\left(\frac{v\cdot \nabla f^{\star}(x)}{\lambda}\right)dx\leq |G^{\star}|\right\}\nonumber\\
&=&\inf\left\{\lambda> 0:\int^{D_{G^{\ast}}/2}_{0}\int_{S^{n-1}}\Phi\left(\frac{|v\cdot u|\cdot |\nabla f^{\star}(ru)|}{\lambda}\right)r^{n-1}du dr\leq |G^{\star}|\right\}\nonumber\\
&\geq&\inf\left\{\lambda> 0:n\omega_{n}\int^{D_{G^{\ast}}/2}_{0}\Phi\left(\frac{1}{n\omega_n}\frac{\int_{S^{n-1}}|v\cdot u| du}{\lambda}|\nabla f^{\star}(ru)|\right)r^{n-1} dr\leq |G^{\star}|\right\}\nonumber\\
&=&\inf\left\{\lambda> 0:n\omega_{n}\int^{D_{G^{\ast}}/2}_{0}\Phi\left(\frac{1}{n\omega_n}\frac{2\omega_{n-1}}{\lambda}|\nabla f^{\star}(ru)|\right) r^{n-1}dr\leq |G^{\star}|\right\}\nonumber\\
&=&\inf\left\{\lambda> 0: \int^{D_{G^{\ast}}/2}_{0}\int_{S^{n-1}}\Phi\left(\frac{2\omega_{n-1}}{n\omega_n}\frac{|\nabla f^{\star}(ru)|}{\lambda}\right)r^{n-1}du dr\leq |G^{\star}|\right\}\nonumber\\
&=&\frac{2\omega_{n-1}}{n\omega_n}\|\nabla f^{\star}\|_{\Phi}.
\end{eqnarray}
Thus
\begin{eqnarray}\label{61}
\|v\|_{f^{\star},\phi}\geq \frac{2\omega_{n-1}}{n\omega_n}\|\nabla f^{\star}\|_{\Phi}.
\end{eqnarray}
Because that $\phi$ is strictly increasing on $[0,\infty)$, thus by the second equality of (\ref{66}), for any $f\in W_0^{n,\Phi}(G)$,
\begin{eqnarray}\label{61a}
\|v\|_{f^{\star},\phi}<\|\nabla f^{\star}\|_{\Phi}.
\end{eqnarray}
By (\ref{61}) and (\ref{61a}), for any $f\in W_0^{n,\Phi}(G)$,
\begin{eqnarray}\label{68}
\frac{2\omega_{n-1}}{n\omega_n}\leq \frac{\|v\|_{f^{\star},\phi}}{\|\nabla f^{\star}\|_{\Phi}}<1.
\end{eqnarray}

Combining (\ref{1a}) and
(\ref{68}) gives the desired inequality.
\end{proof}

\begin{lem}\label{7.2}
If $f\in  W_0^{n,\Phi}(G)$, then $\mathcal{E}_{\phi}(f)/\|\nabla f\|_{\Phi}$ is  uniformly bounded from above by a positive constant, i.e.,
\begin{eqnarray}\label{18e}
\sup\left\{\frac{\mathcal{E}_{\phi}(f)}{\|\nabla f\|_{\Phi}}:\;\;f\in  W_0^{n,\Phi}(G) \right\}\leq \omega_n^{-\frac{1}{n}}.
\end{eqnarray}
Moreover, $\mathcal{E}_{\phi}(f)/\|\nabla f\|_{\Phi}$ is not uniformly bounded from below by any positive constant, i.e.,
\begin{eqnarray}
\inf\left\{\frac{\mathcal{E}_{\phi}(f)}{\|\nabla f\|_{\Phi}}:\;\;f\in  W_0^{n,\Phi}(G) \right\}=0.
\end{eqnarray}
\end{lem}
\begin{proof}
On the one hand, since $\phi$ is even and strictly increasing on $[0,\infty)$,
\begin{eqnarray}\label{18a}
\|v\|_{f,\phi}&=&\inf\left\{\lambda>0:\int_{G}\phi\left(\frac{|v\cdot u||\nabla f(x)|}{\lambda}\right)dx\leq |G|\right\}\nonumber\\
&\leq&\inf\left\{\lambda>0:\int_{G}\Phi\left(\frac{|\nabla f(x)|}{\lambda}\right)dx\leq |G|\right\}\nonumber\\
&=&\|\nabla f\|_{\Phi}.
\end{eqnarray}
By (\ref{1a}) and (\ref{18a})
$$\mathcal{E}_{\phi}(f)=\left(\frac{1}{n}\int_{S^{n-1}}\|v\|^{-n}_{f,\phi}dv\right)^{-\frac{1}{n}}\leq \omega_n^{-\frac{1}{n}}\|\nabla f\|_{\Phi}.$$
Thus, (\ref{18e}) is established.

On the other hand, let $f_1(x)=f(Ax)$, where $A\in SL(n)$. By Lemma \ref{9c}, $\mathcal{E}_{\phi}(f)=\mathcal{E}_{\phi}(f_1)$. Let
\begin{equation}
A=\left(\begin{matrix}
\tau&0&0&\cdots&0&\\
0&\frac{1}{\tau}&0&\cdots&0&\\
0&0&1&\cdots&0&\\
&&&\ddots&&\\
0&0&0&\cdots&1&
\end{matrix}\right),
\end{equation}
where $\tau>0$.   Let $D(f_1)$ denote the domain of $f_1$. Then $D(f_1)=A^{-1}G$.
Let $\eta=(\tau,1/\tau,1,\cdots,1)$ be a $n$-dimensional vector,  we have
\begin{eqnarray}
\|\nabla f_1\|_{\Phi}&=&\inf\left\{\lambda>0:\;\frac{1}{|A^{-1}G|}\int_{A^{-1}G}\Phi\left(\frac{|A^t\nabla f(Ax)|}{\lambda}\right)dx\leq1\right\}\nonumber\\
&=&\inf\left\{\lambda>0:\;\frac{1}{|G|}\int_{G}\Phi\left(\frac{|\eta\cdot\nabla f(x^{\prime})|}{\lambda}\right)|A^{-1}|dx^{\prime}\leq1\right\}.
\end{eqnarray}

Since $\phi$ is strictly monotone increasing on $(0,\infty)$,  $\|\nabla f_1\|_{\Phi}\rightarrow \infty$ when $\tau\rightarrow \infty$, which implies
$$\inf\left\{\frac{\mathcal{E}_{\phi}(f)}{\|\nabla f\|_{\Phi}}:\;\;f\in  W^{n,\Phi}(G) \right\}=0.$$
\end{proof}

By Theorem \ref{1}, Lemma \ref{7.1} and Lemma \ref{7.2}, we have
\begin{eqnarray}\label{7.5}
\frac{2\omega_{n-1}}{n\omega_n^{(n+1)/n}}\|\nabla f^{\star}\|_{\Phi}\leq \mathcal{E}_{\phi}(f^{\star})\leq \mathcal{E}_{\phi}(f)\leq \omega^{-\frac{1}{n}}_n\|\nabla f\|_{\Phi}.
\end{eqnarray}
By (\ref{7.5}), the following theorem is established.
\begin{thm}\label{7.3}
For $f\in W_0^{n,\Phi}(G)$, if
$$\mathcal{E}_{\phi}(f^{\star})\leq \mathcal{E}_{\phi}(f),$$
then
$$\|\nabla f^{\star}\|_{\Phi}\leq c_1\|\nabla f\|_{\Phi},$$
where $c_1=\frac{n\omega_n}{2\omega_{n-1}}$ is a constant.
\end{thm}

By Lemma \ref{7.2} and Theorem \ref{7.3}, the affine Orlicz P\'olya-Szeg\" o inequality (\ref{a1}) is essentially stronger than the Euclidean Orlicz P\'olya-Szeg\" o inequality (\ref{1d}).

\noindent\small  School of Mathematics and Statistics, Chongqing Technology and Business University,\\
\small  Chongqing 400067, PR China\\
E-mail address: lxyoujiang@126.com\\
\small  Department of Mathematics, Tandon School of Engineering, New York University,\\
\small  6 MetroTech Center, Brooklyn, NY 11201, USA\\
E-mail address: yjl432@nyu.edu


\begin{thebibliography}{MRY90}


\bibitem{AFP00} L. Ambrosio, N. Fusco, D. Pallara, {\it Functions of Bounded Variation and Free Discontinuity
Problems}, Oxford University Press, Oxford, 2000.

\bibitem{BCFP} M. Barchiesi, G.M. Capriani, N. Fusco, G. Pisante, {\it  Stability of P\'olya-Szeg\"o inequality for log-concave functions}, J. Funct. Anal. 267 (2014), 2264-2297.

\bibitem{Bianchi17} G. Bianchi, R. J. Gardner, P. Gronchi, {\it Symmetrization in geometry}, Adv. Math. 306 (2017), 51-88.


\bibitem{Boroczky13} K.J. B\"or\"oczky, {\it Stronger versions of the Orlicz-Petty projection inequality}, J. Differential Geom. 95 (2013), 215-247.

\bibitem{BLYZ12} K.J. B\"or\"oczky, E. Lutwak, D. Yang, G. Zhang, {\it The log-Brunn-Minkowski inequality}, Adv. Math. 231 (2012), 1974-1997.


\bibitem{BZ88} J.E. Brothers, W.P. Ziemer, {\it Minimal rearrangements of Sobolev functions}, J. Reine Angew. Math. 384 (1988), 153-179.

\bibitem{Bu97} A. Burchard, {\it Steiner symmetrization is continuous in $W^{1,p}$}, Geom.
Funct. Anal. 7 (1997), 823-860.

\bibitem{Bu04} A. Burchard, Y. Guo, {\it Compactness via symmetrization}, J. Funct. Anal. 214 (2004), 40-73.

\bibitem{BF15} A. Burchard, A. Ferone, {\it On the Extremals of the P\'olya-Szeg\"o Inequality}, Indiana Univ. Math. J. 64 (2015), 1447-1463.


\bibitem{Capriani14} G.M. Capriani, {\it The Steiner rearrangement in any codimension}, Calc. Var. Partial Differential Equations 49 (2014), 517-548.


\bibitem{CCF05}  M. Chleb\'{i}k, A. Cianchi, N. Fusco, {\it The perimeter inequality under
Steiner symmetrization: cases of equality}, Ann. of Math. 162
(2005), 525-555.

\bibitem{CPS15} A. Cianchi, L. Pick, L. Slav¨ªkov\'a {\it Higher-order Sobolev embeddings and isoperimetric inequalities}, Adv. Math. 273 (2015), 568-650.

\bibitem{Ci00} A. Cianchi, {\it Second-order derivatives and rearrangements}, Duke Math. J. 105 (2000), 355-385.
%
\bibitem{CEFT08} A. Cianchi, L. Esposito, N. Fusco, C. Trombetti, {\it A quantitative P\'olya-Szeg\"o principle}, J. Reine
Angew. Math. 614 (2008), 153-189.


\bibitem{CF06} A. Cianchi, N. Fusco, {\it Steiner symmetric extremals in P\'olya-Szeg\"o type inequalities},  Adv. Math. 203 (2006), 673-728.

\bibitem{CF0602} A. Cianchi, N. Fusco, {\it Minimal rearrangements, strict convexity and critical points}, Appl. Anal. 85 (2006), 67-85.



\bibitem{CLYZ09} A. Cianchi, E. Lutwak, D. Yang, G. Zhang, {\it Affine Moser-Trudinger and Morrey-Sobolev inequalities}, Calc. Var. Partial Differential Equations 36 (2009), 419-436.

\bibitem{Ci10} A. Cianchi,  {\it On some aspects of the theory of Orlicz-Sobolev spaces}. In {\it  Around the research of Vladimir Maz'ya. I,} volume 11 of {\it  Int. Math. Ser. (N. Y.)},  81-104. Springer, New York, 2010.

\bibitem{ET04} L. Esposito, C. Trombetti, {\it Convex symmetrization and P\'olya-Szeg\"o inequality}, Nonlinear Anal. 56 (2004), 43-62.

\bibitem{EG92} L.C. Evans, R.F. Gariepy, {\it Measure Theory and Fine Properties of Functions}, Studies in
Advanced Math., CRC Press, 1992.

\bibitem{Evans90} L.C. Evans, {\it Weak Convergence Methods for Nonlinear Partial Differential Equations},
CBMS Regional Conference Series in Mathematics 74, American Mathematical Society,
Providence, RI, 1990.

\bibitem{FMP08} N. Fusco, F. Maggi, A. Pratelli, {\it The sharp quantitative isoperimetric inequality},  Ann. of Math. 168 (2008), 941-980.

\bibitem{Gardner02} R.J. Gardner, {\it The Brunn-Minkowski inequality}, Bull. Amer. Math. Soc. 39 (2002), 355-405.

\bibitem{Ga06} R.J. Gardner, {\it Geometric Tomography}, second edition, Cambridge University
Press, New York, 2006.

\bibitem{GHW14} R.J. Gardner, D. Hug, W. Weil, {\it The Orlicz-Brunn-Minkowski theory: a general framework, additions, and inequalities}, J. Differential Geom. 97 (2014), 427-476.

\bibitem{GHW15} R.J. Gardner, D. Hug, W. Weil, D. Ye, {\it The dual Orlicz-Brunn-Minkowski theory}, J. Math. Anal. Appl. 430 (2015), 810-829.

\bibitem{GZ98} R.J. Gardner, G. Zhang, {\it Affine inequalities and radial mean bodies}, Amer. J. Math. 120 (1998), 505-528.

\bibitem{Gruber07} P.M. Gruber, {\it Convex and Discrete Geometry}, Springer, Berlin, 2007.

\bibitem{HLYZ10} C. Haberl, E. Lutwak, D. Yang, G. Zhang, {\it The even Orlicz Minkowski problem}, Adv. Math. 224 (2010), 2485-2510.
%
\bibitem{HS09} C. Haberl, F.E. Schuster, {\it General $L_p$ affine isoperimetric inequalities}, J. Differential Geom. 83 (2009), 1-26.

\bibitem{HS0902} C. Haberl, F.E. Schuster, {\it Asymmetric affine $L_p$ Sobolev inequalities}, J. Funct. Anal. 257 (2009), 641-658.



\bibitem{HSX12} C. Haberl, F.E. Schuster, J. Xiao, {\it An asymmetric affine P\'olya-Szeg\"o principle}, Math. Ann. 352 (2012), 517-542.

\bibitem{HP14} C. Haberl, L. Parapatits, {\it The centro-affine Hadwiger theorem}, J. Amer. Math. Soc. 27 (2014), 685-705.

\bibitem{HJM16} J. Haddad, C.H. Jim\'enez, M. Montenegro, {\it Sharp affine Sobolev type inequalities via the $L_p$ Busemann-Petty centroid inequality}, J. Funct. Anal. 271 (2016), 454-473.

\bibitem{LL11} A.-J. Li, D. Xi, G. Zhang, {\it Volume inequalities of convex bodies from cosine transforms on Grassmann manifolds}, Adv. Math. 304 (2017), 494-538.

\bibitem{Lijin16} J. Li, G. Leng, {\it Orlicz valuation}, Indiana Univ. Math. J. 66 (2017), 791-819.

\bibitem{Lin17} Y. Lin, {\it Smoothness of the Steiner symmetrization}, to appear in Proc. Amer. Math. Soc. Available at http://dx.doi.org/10.1090/proc/13683.

\bibitem{LYJ17} Y. Lin, {\it Affine Orlicz P\'olya-Szeg\"o principle for log-concave functions}, J. Funct. Anal. 273 (2017), 3295-3326.

\bibitem{Ludwig02} M. Ludwig, {\it Projection bodies and valuations}, Adv. Math. 172 (2002), 158-168.

\bibitem{Ludwig05} M. Ludwig, {\it Minkowski valuations}, Trans. Amer. Math. Soc. 357 (2005), 4191-4213.

\bibitem{Ludwig14} M. Ludwig, {\it Anisotropic fractional Sobolev norms}, Adv. Math. 252 (2014), 150-157.

\bibitem{LR10} M. Ludwig, M. Reitzner, {\it A classification of $SL(n)$ invariant valuations}, Ann. of Math. 172 (2010), 1219-1267.

\bibitem{LXZ11} M. Ludwig, J. Xiao, G. Zhang, {\it Sharp convex Lorentz-Sobolev inequalities}, Math. Ann. 350 (2011), 169-197.

\bibitem{Lu93} E. Lutwak, {\it The Brunn-Minkowski-Firey theory. I. Mixed volumes and the Minkowski problem}, J. Differential Geom. 38 (1993), 131-150.

\bibitem{Lu96} E. Lutwak, {\it The Brunn-Minkowski-Firey theory. II. Affine and geominimal surface areas}, Adv. Math. 118 (1996), 244-294.

\bibitem{LYZ00} E. Lutwak, D. Yang, G. Zhang, {\it $L_p$ affine isoperimetric inequalities}, J. Differential Geom. 56 (2000), 111-132.

\bibitem{LYZ02} E. Lutwak, D. Yang, G. Zhang, {\it Sharp affine $L_p$ Sobolev inequalities}, J. Differential Geom. 62 (2002), 17-38.

\bibitem{LYZ06} E. Lutwak, D. Yang, G. Zhang, {\it Optimal Sobolev norms and the $L^p$ Minkowski problem}, Int. Math. Res. Not. (2006), Art. ID 62987, 1-21.

\bibitem{LYZ10} E. Lutwak, D. Yang, G. Zhang, {\it Orlicz projection bodies}, Adv. Math. 223 (2010), 220-242.

\bibitem{LYZ1002} E. Lutwak, D. Yang, G. Zhang, {\it Orlicz centroid bodies}, J. Differential Geom. 84 (2010), 365-387.

\bibitem{Ma85} V.G. Maz'ya, {\it Sobolev Spaces}, Springer-Verlag, Berlin (1985).

\bibitem{Ma11} V.G. Maz'ya, {\it Sobolev Spaces with Applications to Elliptic Partial Differential Equations}, Springer, Heidelberg, 2011.

\bibitem{Nguyen16} V.H. Nguyen, {\it New approach to the affine P\'olya-Szeg\"o  principle and the stability version of the affine Sobolev inequality}, Adv. Math. 302 (2016), 1080-1110.

\bibitem{PS51} G. P\'olya, G. Szeg\"o, {Isoperimetric inequalities in mathematical physics}, Ann. Math. Stud. 27, Princeton University Press (1951).


\bibitem{Schneider13} R. Schneider, {\it Convex Bodies: The Brunn-Minkowski Theory}, Encyclopedia of Mathematics and its Applications, 151. Cambridge Univ. Press, Cambridge, 2014.

\bibitem{Ta94} G. Talenti, {\it Inequalities in rearrangement invariant function spaces}, In: M. Krbec, A. Kufner, B. Opic, J. R\'akosnik (eds.) Nonlinear Analysis, Function Spaces and Applications, vol. 5, 177-230, Prometheus, Prague (1994).

\bibitem{Trudinger97} N.S. Trudinger, {\it On new isoperimetric inequalities and symmetrization}, J. reine angew. Math. 488 (1997), 203-220.

\bibitem{Wang12} T. Wang, {\it The affine Sobolev-Zhang inequality on $BV(\mathbb{R}^n)$}, Adv. Math. 230 (2012), 2457-2473.

\bibitem{Wang13} T. Wang, {\it The affine P\'olya-Szeg\"o principle:
Equality cases and stability}, J. Funct. Anal. 265 (2013), 1728-1748.

\bibitem{Werner08} E. Werner, D. Ye, {\it New $L_p$ affine isoperimetric inequalities}, Adv. Math. 218 (2008), 762-780.


\bibitem{XJL14} D. Xi, H. Jin, G. Leng, {\it The Orlicz Brunn-Minkowski inequality}, Adv. Math. 260 (2014), 350-374.

\bibitem{Ye16}  D. Ye, {\it Dual Orlicz-Brunn-Minkowski theory: dual Orlicz $L_{\phi}$ affine and geominimal surface areas}. J. Math. Anal. Appl. 443 (2016), 352-371.

\bibitem{Zhang99} G. Zhang, {\it The affine Sobolev inequality}, J. Differential Geom. 53 (1999), 183-202.

\bibitem{ZXZ14} B. Zhu, J. Zhou,  W. Xu, {\it Dual Orlicz-Brunn-Minkowski theory}, Adv. Math. 264 (2014), 700-725.

\bibitem{Zhu12} G. Zhu, {\it The Orlicz centroid inequality for star bodies}, Adv. in App. Math. 48 (2012), 432-445.

\bibitem{ZX14}D. Zou, G. Xiong, {\it Orlicz-John ellipsoids}, Adv. Math. 265 (2014), 132-168.

\bibitem{Ziemer89} W.P. Ziemer, {\it Weakly Differentiable Functions: Sobolev spaces and functions of bounded variation}, GTM 120, Springer Verlag, 1989.
\end{thebibliography}
\end{document}